\documentclass[sn-mathphys-num]{sn-jnl}

\usepackage{graphicx}%
\usepackage{multirow}%
\usepackage{amsmath,amssymb,amsfonts}%
\usepackage{amsthm}%
\usepackage{mathrsfs}%
\usepackage[title]{appendix}%
\usepackage{xcolor}%
\usepackage{textcomp}%
\usepackage{manyfoot}%
\usepackage{booktabs}%
\usepackage{algorithm}%
\usepackage{algorithmicx}%
\usepackage{algpseudocode}%
\usepackage{listings}%

\usepackage{epsfig,amsmath,amsfonts,latexsym,mathrsfs,amssymb,flafter}
\usepackage{palatino}
\usepackage{hyperref}
\usepackage{amsmath,amssymb,amsfonts,amscd,verbatim,graphicx}
\usepackage{charter}
\usepackage{xypic}
\usepackage{graphicx}
\usepackage{stmaryrd}
\usepackage[all]{xy}
\usepackage{amsmath,array}
\usepackage{pst-node}
\usepackage{tikz-cd} 
\usepackage{pifont}
\usepackage[latin1]{inputenc}
\numberwithin{equation}{section}

\theoremstyle{thmstyleone}%
\theoremstyle{thmstyletwo}%

\theoremstyle{thmstylethree}%

\newtheorem{thm}{Theorem}[section]
\newtheorem{prop}[thm]{Proposition}
\newtheorem{lem}[thm]{Lemma}
\newtheorem{cor}[thm]{Corollary}

\theoremstyle{definition}
\newtheorem{defn}[thm]{Definition}
\newtheorem{rmk}[thm]{Remark}
\newtheorem{ex}[thm]{Example}

\newcommand{\R}{\mathbb R}
\newcommand{\Z}{\mathbb Z}
\newcommand{\A}{\mathcal A}

\newcommand{\C}{\mathbb C}
\newcommand{\HH}{\mathbb H}
\newcommand{\G}{\mathcal G}

\newcommand{\X}{\mathcal X}
\newcommand{\M}{\mathcal M}

\newcommand{\ttt}{\mathbf t}

\newcommand{\f}{\mathbf f}

\newcommand{\im}{\mathrm {im}}

\newcommand{\g}{\mathbf g}

\newcommand{\ad}{\mathrm {ad}}

\newcommand{\tr}{\mathrm {Tr}}

\begin{document}

\title[Article Title]{A New Gauge-Theoretic Construction of the 4-Dimensional Hyperk{\"a}hler ALE Spaces}


\author{\fnm{Jiajun} \sur{Yan}}\email{jy4au@virginia.edu; https://orcid.org/0000-0002-5682-3693}

\affil{\orgdiv{Department of Mathematics}, \orgname{University of Virginia}, \orgaddress{\street{141 Cabell Drive}, \city{Charlottesville}, \postcode{22903}, \state{VA}, \country{USA}}}


\abstract{Non-compact hyperk{\"a}hler spaces arise frequently in gauge theory. The 4-dimensional hyperk{\"a}hler ALE spaces are a special class of complete non-compact hyperk{\"a}hler spaces. They are in one-to-one correspondence with the finite subgroups of SU(2) and have interesting connections with representation theory and singularity theory, captured by the McKay Correspondence.
The 4-dimensional hyperk{\"a}hler ALE spaces are first classified by Peter Kronheimer via a finite-dimensional hyperk{\"a}hler reduction. In this paper, we give a new gauge-theoretic construction of these spaces. More specifically, we realize each 4-dimensional hyperk{\"a}hler ALE space as a moduli space of solutions to a system of equations for a pair consisting of a connection and a section of a vector bundle over an orbifold Riemann surface, modulo a gauge group action. The construction given in this paper parallels Kronheimer's original construction and hence can also be thought of as a gauge-theoretic interpretation of Kronheimer's construction of these spaces.}



\pacs[Statements and Declarations]{(1) Conflicts of Interest Statement: The author declares that there are no conflicts of interest related to the manuscript submitted for publication; (2) Data Availability Statement: The author declares that the manuscript has no associated data.}

\pacs[Acknowledgements]{This paper consists of the main results of the author's Ph.D. thesis. The author would like to thank her advisor Thomas Mark for suggesting this project as well as his patient guidance and long-term support throughout the duration of the project. The author would also like to thank Sergey Cherkis, Lorenzo Foscolo, Aleksander Doan, Richard Wentworth and Saman Habibi Esfahani for their helpful comments and feedback.}

\maketitle

\section{Introduction}

In this paper, we give a new gauge-theoretic construction of all the $4$-dimensional hyperk{\"a}hler ALE (asymptotically locally Euclidean) spaces. These spaces are originally constructed by Peter Kronheimer in his thesis  \cite{kronheimer}. They are in one-to-one correspondence with the finite subgroups of $SU(2)$ and have deep connections with representation theory, singularity theory and low-dimensional topology. Topologically, these spaces are plumbed $4$-manifolds where the plumbing graphs are described by the ADE-type Dynkin diagrams of semi-simple Lie algebras. Geometrically, they are each a resolution of singularity of $\C^2/\Gamma$, where $\Gamma$ is a finite subgroup of $SU(2)$ and the blowup diagram naturally corresponds to the plumbing graph. The interesting connections these spaces share with representation theory, singularity theory and low-dimensional topology are captured by the McKay Correspondence \cite{mckay}. In Kronheimer's construction, each of them is realized through a hyperk{\"a}hler reduction of a finite-dimensional vector space. We will review this construction in Section 2.  

On the other hand, non-compact hyperk{\"a}hler spaces frequently arise in gauge theory as the moduli spaces of solutions to gauge-theoretic equations. Well-known examples include but are not limited to the Hitchin moduli spaces of solutions to self-duality equations on Riemann surfaces \cite{hitchin}, and moduli spaces of monopoles \cite{atiyah-hitchin}, \cite{cherkis}, \cite{foscolo1}, \cite{foscolo2}. A typical approach to construct hyperk{\"a}hler moduli spaces is to realize them via hyperk{\"a}hler reduction. A special class of complete non-compact hyperk{\"a}hler manifolds is known as the gravitational instantons. There are various gauge-theoretic constructions of these hyperk{\"a}hler manifolds such as the ones given in the following (incomplete) list of papers: \cite{atiyah-hitchin},  \cite{boalch}, \cite{cherkis}, \cite{cherkis-k}, \cite{cherkis-w}, \cite{nahm}.  The $4$-dimensional hyperk{\"a}hler ALE spaces are also gravitational instantons and here we give a new construction of these spaces using a gauge-theoretic approach. Our construction shares a strong parallel to Kronheimer's original construction  \cite{kronheimer} which is not gauge-theoretic, and ergo, our construction can be thought of as a gauge-theoretic interpretation of Kronheimer's construction. More specifically, we realize each  $4$-dimensional hyperk{\"a}hler ALE space as the moduli space of solutions to a system of equations for a pair consisting of a connection and a section of a vector bundle over an orbifold Riemann surface modulo a gauge group action, as stated in the following theorem. The precise definitions of the notations in the theorem below will be given in Section 2 and Section 3.

\begin{thm} \label{main}
Let $\tilde{\zeta}=(\tilde{\zeta}_1,\tilde{\zeta}_2,\tilde{\zeta}_3)$, where for all $i$, $\tilde{\zeta}_i\in Z$. Let $$\X_{\tilde{\zeta}}=\{(B,\Theta)\in \A^F_\tau \times C^\infty(S^2/\Gamma, E(\Gamma)) | (\ref{holo} ) - (\ref{mu3})\}/\G^{F,\Gamma}_\tau.$$ Then for a suitable choice of $\tilde{\zeta}$, $\X_{\tilde{\zeta}}$ is diffeomorphic to the resolution of singularity $\widetilde{\C^2/\Gamma}$. Furthermore, $\X_{\tilde\zeta}$ is isometric to  $X_{\zeta}$ in \cite{kronheimer}.
\end{thm}

Below is the layout of the paper. In Section 2, we give a review of Kronheimer's original construction of ALE spaces \cite{kronheimer} and introduce various notations. In Section 3, we give an overview of the new gauge-theoretic construction of ALE spaces where we write down the bundles and gauge groups that we will working with for the construction. We also sketch the procedures for constructing the moduli spaces. In Section 4, we construct hyperk{\"a}hler structures on some infinite-dimensional spaces which will be important for the construction. Section 5 and Section 6 deal with the main technical steps for proving the main theorem (cf. Theorem \ref{main}) and we prove the main theorem in Section 7.  

We emphasize here that all the hyperk{\"a}hler metrics appearing in this paper are \emph{complete}.

\section{Preliminaries}

We begin the section with a review of Kronheimer's construction of ALE spaces in \cite{kronheimer} which will be of great importance throughout the paper. Then, we will lay out the basic setup for the main gauge-theoretic construction, introducing definitions central to the construction as well as fixing notations and conventions.

\subsection{Kronheimer's construction of ALE spaces}

We review Kronheimer's construction of ALE spaces via hyperk{\"a}hler reduction in \cite{kronheimer} in this subsection. 

Let $\Gamma$ be a finite subgroup of $SU(2)$ and let $R$ be its regular representation. Let $Q\cong\C^2$ be the canonical $2$-dimensional representation of $SU(2)$ and let $P=Q\otimes_\C End(R)$, where $End(R)$ denote the endomorphism space of $R$. Let $M=P^\Gamma$ be the space of $\Gamma$-invariant elements in $P$. After fixing a $\Gamma$-invariant hermitian metric on $R$, $P$ and $M$ can be regarded as right $\HH$-modules. Now, choose an orthonormal basis on $Q$, then we can write an element in $P$ as a pair of matrices $(\alpha,\beta)$ with $\alpha,\beta\in End(R)$, and the action of $J$ on $P$ is given by $$J(\alpha,\beta)=(-\beta^*,\alpha^*).$$ Since the action of $\Gamma$ on $P$ is $\HH$-linear, the subspace $M$ is then an $\HH$-submodule, which can be regarded as a flat hyperk{\"a}hler manifold. Explicitly, a pair $(\alpha,\beta)$ is in $M$ if for each $$\gamma=\begin{pmatrix}
u & v \\
-v^* & u^*  
\end{pmatrix} \in \Gamma,$$ where $v^*$ and $u^*$ denote the complex conjugate of $v$ and $u$, respectively, we have 
\begin{equation}\label{1}R(\gamma^{-1})\alpha R(\gamma)=u\alpha+v\beta, \end{equation}  \begin{equation}\label{2}R(\gamma^{-1})\beta R(\gamma)=-v^*\alpha+u^*\beta.\end{equation}

Let $U(R)$ denote the group of unitary transformations of $R$ and let $F$ be the subgroup formed by elements in $U(R)$ that commute with the $\Gamma$-action on $R$. The natural action of $F$ on $P$ is given by the following: for $f\in F$, $$(\alpha,\beta)\mapsto (f\alpha f^{-1},f\beta f^{-1}).$$ 

Again, the action of $F$ on $P$ is $\HH$-linear and preserves $M$. On the other hand, since $F$ acts by conjugation, the scalar subgroup $T\subset F$ acts trivially, and hence, we get an action of $F/T$ on $M$ that preserves $I$, $J$, $K$. 

Now, let $\f/\ttt$ be the Lie algebra of $F/T$ and identify $(\f/\ttt)^*$ with the traceless elements of $\f\subset End(R)$.  As the action of $F/T$ on $M$ is Hamiltonian with respect to $I$, $J$, $K$, we obtain the following moment maps:
$$\mu_1(\alpha,\beta)=\frac{i}{2}([\alpha,\alpha^*]+[\beta,\beta^*]),$$
$$\mu_2(\alpha,\beta)=\frac{1}{2}([\alpha,\beta]+[\alpha^*,\beta^*]),$$
$$\mu_3(\alpha,\beta)=\frac{i}{2}(-[\alpha,\beta]+[\alpha^*,\beta^*]).$$

Let $\mu=(\mu_2,\mu_2,\mu_3): M\to \R^3\otimes(\f/\ttt)^*$. Let $Z$ denote the center of $(\f/\ttt)^*$ and let $\zeta=(\zeta_1,\zeta_2,\zeta_3)\in \R^3\otimes Z$. For $\zeta$ lying in the ``good set", we get $X_\zeta=\mu^{-1}(\zeta)/F$ is a smooth $4$-manifold diffeomorphic to $\widetilde{\C^2/\Gamma}$.

\begin{prop}[cf. Proposition 2.1. in \cite{kronheimer}]\label{hkred}
Suppose that $F$ acts freely on $\mu^{-1}(\zeta)$. Then \begin{enumerate}
\item $d\mu$ has full rank at all points of $\mu^{-1}(\zeta)$, so that $X_\zeta$ is a nonsingular manifold of $\dim M- 2\dim F$ (\emph{resp.}  $\dim M- 4\dim F$),
\item the metric $g$ and complex structures $I$ (\emph{resp.} $I$, $J$, $K$) descend to $X_\zeta$, and equipped with these, $X_\zeta$ is K{\"a}hler (\emph{resp.} hyperk{\"a}hler). \end{enumerate}
 \end{prop}

Now, we review some basic representation theory regarding to the McKay Correspondence mentioned in \cite{kronheimer}. Let $R_0,...,R_r$ be the irreducible representations of $\Gamma$ with $R_0$ the trivial representation, and let $$Q\otimes R_i=\bigoplus_j a_{ij}R_j$$ be the decomposition of $Q\otimes R_i$ into irreducibles. The representations $R_1,...,R_r$ correspond to the set of simple roots $\xi_1,...,\xi_r$ for the associated root system of one of the ADE-type Dynkin diagrams. Furthermore, if $\xi_0=-\sum_1^rn_i\xi_i$ is the negative of the highest root, then we have that for all $i$, $$n_i=\dim R_i.$$ Hence, the regular representation $R$ decomposes as $$R=\bigoplus_i \C^{n_i}\otimes R_i,$$ and $M$ decomposes as $$M=\bigoplus_{i,j}a_{ij}Hom(\C^{n_i},\C^{n_j}),$$ and $F$ can be written as $$F=\times_iU(n_i).$$ Consequently, we get $$\dim_\R M=\sum_{i,j}2a_{ij}n_in_j=\sum_i4n_i^2=4|\Gamma|,$$ and $$\dim_\R F=\sum_in_i^2=|\Gamma|.$$

The center of the Lie algebra $\f$ is spanned by the elements $\sqrt{-1}\pi_i$, where $\pi_i$ is the projection $\pi_i:R\to \C^{n_i}\otimes R_i$ ($i=0,...,r$). Let $h$ be the real Cartan algebra associated to the Dynkin diagram, then there is a linear map $l$ from the center of $\f$ to $h^*$ defined by the following: $$l:\sqrt{-1}\pi_i\mapsto n_i\xi_i.$$ The kernel of $l$ is the one-dimensional subalgebra $\ttt\subset\f$, so on the dual space, we get an isomorphism $$\iota: Z \to h.$$ For each root $\xi$, we write $$D_\xi=\ker(\xi\circ\iota).$$

\begin{prop}[cf. Proposition 2.8. in \cite{kronheimer}]
If $F/T$ does not act freely on $\mu^{-1}(\zeta)$, then $\zeta$ lies  in one of the codimensional-$3$ subspaces $\R^3\otimes D_\xi\subset \R\otimes Z$, where $\xi$ is a root. 
\end{prop} 

Hence, the ``good set" mentioned earlier in the subsection refers to the following:  $$(\R^3\otimes Z)^\circ=(\R^3\otimes Z)\setminus \bigcup_\xi(\R^3\otimes D_\xi).$$

The following theorems are also proven in \cite{kronheimer} and \cite{kronheimer2}, and together, they give a complete construction and classification of ALE spaces. For all the theorems below in this subsection, let $(X,g)$ be a $4$-dimensional hyperk{\"a}hler manifold. 

\begin{thm}[cf. Theorem 1.1. in \cite{kronheimer}]
Let three cohomology classes $\alpha_1, \alpha_2, \alpha_3\in H^2(X;\R)$ be given which satisfy the nondegeneracy condition: 

$(*)$ For each $\Sigma\in H_2(X;\Z)$ with $\Sigma\cdot\Sigma=-2$, there exists $i\in\{1,2,3\}$ with $\alpha_i(\Sigma)\neq 0$.
 
Then there exists on $X$ an ALE hyperk{\"a}hler structure for which the cohomology classes of the K{\"a}hler form $[\omega_i]$ are the given $\alpha_i$.
\end{thm}

\begin{thm}[cf. Theorem 1.2. in \cite{kronheimer}]
Every ALE hyperk{\"a}hler $4$-manifold is diffeomorphic to the minimal resolution of $\C^2/\Gamma$ for some $\Gamma\subset SU(2)$, and the cohomology classes of the K{\"a}hler forms on such a manifold must satisfy the nondegeneracy condition $(*)$.
\end{thm}

\begin{thm}[cf. Theorem 1.3. in \cite{kronheimer}]
If $X_1$ and $X_2$ are two ALE hyperk{\"a}hler $4$-manifolds, and there is a diffeomorphism $X_1\to X_2$ under which the cohomology classes of the K{\"a}hler forms agree, then $X_1$ and $X_2$ are isometric.
\end{thm}

\subsection{Basic setup for the gauge-theoretic construction}
We start off by considering $S^3$ as a principal $S^1$-bundle over $S^2$ via the dual Hopf fibration. The explicit construction is as follows:
$$S^3=\{(z_1,z_2)\in \C^2: |z_1|^2+|z_2|^2=1\}.$$ The $S^1$-action on $S^3$ is given by the following: for $g=e^{i\theta} \in S^1$,
$$(z_1,z_2)g=(z_1e^{-i\theta}, z_2e^{-i\theta}).$$ 

If we think of the base $S^2$ as the unit sphere sitting inside $\R^3$, we can write down the projection map  explicitly which will be useful later on: let $\pi: S^3 \to S^2$ be the projection map with $$\pi(z_1,z_2)=(2z_1z_2^*,|z_1|^2-|z_2|^2).$$ In terms of real coordinates, we have $$\pi(a, b, c, d)=(2(ac+bd), bc-ad, a^2+b^2-c^2-d^2).$$ Equivalently, we can think of $\pi$ as a map from $S^3$ to $\C P^1$ given by $\pi: S^3 \to \C P^1$ with $\pi(z_1,z_2)=[z_1:z_2]$.

Now, we turn to the associated bundles of $S^3$. A complex vector space $V$ with an $S^1$-action on $V$ determines a vector bundle over $S^2$ associated to $S^3$ with fiber $V$. Below, we consider the specific $S^1$-action on a complex vector space $V$ given by the usual scalar multiplication.
\begin{defn}
Let $V$ be a complex vector space with $S^1$-action given by scalar multiplication. Then $E(V)$ is defined as $E(V)=S^3 \times V/\sim$, where $[p,v]\sim[pg,g^{-1}v]$, for all $g\in S^1$, and $E(\overline{V})$ is defined as $E(\overline{V})=S^3 \times V/\sim$, where $[p,v]\sim[pg,gv]$, for all $g\in S^1$. 

\end{defn}
There are three important examples that we will be working with closely, i.e., the hyperplane bundle, the tautological bundle and the associated bundle with fiber $V=End(R)$, where $R$ is the regular representation of a finite subgroup $\Gamma$ of $SU(2)$. 

\begin{ex}[The hyperplane bundle and the tautological bundle] 
Let $V=\C$. Then for $g=e^{i\theta}\in S^1$, $p\in S^3$ and $v\in\C$,  we have that $[p,v]\sim [pg,g^{-1}v]=[pe^{-i\theta},e^{-i\theta}v]$. Hence, $E(\C)$ is isomorphic to the hyperplane bundle $H$ over $\C P^1$. Similarly, $E(\overline{\C})$ is isomorphic to the tautological bundle over $\C P^1$. 

\end{ex}

\begin{ex}
Let $\Gamma$ be a finite subgroup of $SU(2)$, and let $R$ be the regular representation of $\Gamma$ with invariant hermitian metric. We see that $E(V)$ splits orthogonally into a direct sum of hyperplane bundles, that is, $E(V)=\oplus_i H_i$, where each $H_i=E(\C\cdot e_i)$ is isomorphic to the hyperplane bundle $H$. \end{ex}

\subsection{The $\Gamma$-action and orbifold vector bundles}
Let $\Gamma$ be a finite subgroup of $SU(2)$ as before, and let $V$ be a representation of $\Gamma$ given by $r: \Gamma \to GL_\C(V)$. We want to build an orbifold vector bundle incorporating the $\Gamma$-representation. To proceed, we introduce the following definition.

\begin{defn}

\begin{enumerate}

\item Suppose either $\Gamma$ doesn't contain $-1\in SU(2)$ or $\Gamma$ contains $-1$ and $r(-1)=-1\in GL_\C(V)$. Let $E(V)^\Gamma_r$ be defined as follows: $E(V)^\Gamma_r=S^3 \times V/\sim$, where $[p,v]\sim[pg,g^{-1}v]\sim[ p\gamma,\gamma^{-1} v]$, for all $g\in S^1$ and $\gamma \in \Gamma$, where $\gamma^{-1}v:=r(\gamma^{-1})v$. 

\item Suppose $\Gamma$ contains $-1\in SU(2)$ and $r(-1) \neq -1 \in GL_\C(V)$, we decompose $V$ into the eigenspaces of $r(-1)$, and write $V=V_0\oplus V_1$, where $r(-1)$ acts as $1$ on $V_0$ and acts as $-1$ on $V_1$. Define  $E(V)^\Gamma_r$ to be $E(V)^\Gamma_r=S^3 \times V_1/\sim$, where $[p,v]\sim[pg,g^{-1}v]\sim[ p\gamma,\gamma^{-1} v]$, for all $g\in S^1$ and $\gamma \in \Gamma$, where $\gamma^{-1}v:=r(\gamma^{-1})v$. 

We will oftentimes abbreviate $E(V)^\Gamma_r$ as $E(\Gamma)$. 
\end{enumerate}

\end{defn}

\begin{rmk}
\begin{enumerate}

\item We can think of $E(\Gamma)$ as an orbifold vector bundle over $S^2/\Gamma$.

\item Let $C^\infty(S^2, E(V))$ denote the space of smooth sections of $E(V)$ and let $C^\infty(S^2, E(V))^\Gamma$ denote the space of $\Gamma$-invariant sections of $E(V)$. With the above definition, we always have that $C^\infty(S^2, E(V))^\Gamma \cong C^\infty_{orb}(S^2/\Gamma, E(V)^\Gamma_r)=C^\infty_{orb}(S^2/\Gamma, E(\Gamma))$. Note that we will begin to drop the subscript and simply use $C^\infty(S^2/\Gamma, E(\Gamma))$ or $C^\infty(E(\Gamma))$ to denote the space of orbifold sections of $E(\Gamma)$ or equivalently the $\Gamma$-invariant sections of $E(V)$ in the coming sections.

\item If we let $V$ be equal to the endomorphism space $End(R)$ of the regular representation $R$ of $\Gamma$, then we have $\gamma^{-1}v=R(\gamma^{-1})vR(\gamma)$. Recall that in Kronheimer's construction, when forming $M=P^\Gamma=(Q\otimes V)^\Gamma$, the element $-1\in\Gamma$ acts on $Q$ by scalar multiplication and on $V$ by the $\Gamma$-representation $r(-1)$. Hence, if an element $\sum q \otimes v$ is $\Gamma$-invariant, we must have $\sum q \otimes v=\sum (-q)\otimes r(-1)v$, so $r(-1)$ must act as $-1\in GL_\C(V)$ on $V$, so $\sum q \otimes v$ lies in $Q \otimes V_1$. In other words, $P^\Gamma=(Q\otimes V_1)^\Gamma$.

\end{enumerate}
\end{rmk}

Now, we want to equip the bundles with a pointwise metric. Let $E(V)$ and $E(\Gamma)$ be defined as above. 

\begin{defn}
\begin{enumerate}
\item Let $h_V$ be a hermitian metric on $V$. Then the pointwise hermitian metric $h_{E(V)}$ on $E(V)$ with respect to $h_V$ is given by $h_{E(V)}([p,v_1],[p,v_2])_x=h_V(v_1,v_2)$, where $p \in S^3$ lies in the fiber over $x \in S^2$.
\item Suppose $h_V$ is also $\Gamma$-invariant, then $h_V$ gives rise to a pointwise hermitian metric on $E(\Gamma)$ again given by $h_{E(V)}([p,v_1],[p,v_2])_x=h_V(v_1,v_2)$, where $p \in S^3$ lies in the fiber over $x \in S^2$.
\end{enumerate}
\end{defn}

\begin{rmk}
With the above definition, we can identify $E(\overline{V})$ with the dual bundle $E(V)^*$, where ${[p,v]}^*([p,v'])=h_V([p,v'],[p,v])=[p,v^*]([p,v'])$, for $[p,v^*] \in E(\overline{V})$, and the metric on $E(\overline{V})$ is given by taking $$h_{E(\overline{V})}([p,v_1^*],[p,v_2^*])_x=h_{\overline{V}}(v_1^*,v_2^*)=h_V(v_2,v_1).$$ On the other hand, we can also express $h_{E(V)}$ in terms of the trace, that is, let $$h_{E(V)}([p,v_1],[p,v_2])_x=\tr([p,v_1],[p,v_2]^*)_x=\tr(v_1v_2^*).$$
As a result, we also get $E(\Gamma)^*$.
\end{rmk}

\subsection{The gauge-theoretic framework}

We are ready to introduce the gauge-theoretic framework in this paper. Many of the definitions introduced in this section take inspiration from classical books and papers in gauge theory given by the following: \cite{atiyah-bott}, \cite{da-silva}, \cite{morgan},  \cite{kobayashi2}, \cite{kobayashi1}, \cite{salamon2}. We will mainly be working with the orbifold vector bundle $E(\Gamma)$ that we have defined previously for the main construction. We fix a holomorphic structure on $E(\C)=H$, and denote it by $\bar{\partial}$. For the remaining of the section, we assume $V=End(S)$ to be the endomorphism space of some $\Gamma$-representation $S$. We fix a $\Gamma$-invariant hermitian structure $h_V$ on $V$ and hence get pointwise metrics on $E(V)$ and $E(\Gamma)$. We take $\omega_{vol}$ to be the Fubini-Study form on $\C P^1$. 

Let $A_0$ be the unique Chern connection on $H$ compatible with the holomorphic structure $\bar{\partial}$ and the hermitian structure descending from $E(V)$. Note that $A_0$ will be $\Gamma$-invariant as it is invariant under $SU(2)$.

Let $P$ be the bundle of automorphisms of $E(V)$. Then $P$ is in fact the trivial bundle $S^2 \times GL_\C(V)$. Now, let $F\subset U(S)$ be the subgroup of unitary transformations of $S$ that commute with the $\Gamma$-action, and let $T$ be the scalar subgroup of $F$. Then we can think of $\tilde{P}$ defined such that $\tilde{P}=S^2\times F/T$ as a subbundle of $P$, as we can think of $F/T$ as lying inside $GL_\C(V)$ by acting on $V$ by conjugation. As $F$ is the subgroup of $U(S)$ with elements that commute with the $\Gamma$-action, we also get that $\tilde{P}^\Gamma$ defined as $\tilde{P}^\Gamma=S^2/\Gamma \times F/T$ is a subbundle of the bundle automorphisms of $E(\Gamma)$. This motivates the following definition. 

\begin{defn}
 Let $V=End(S)$ be the endomorphism space of some $\Gamma$-representation $S$. Let $F\subset U(S)$ be the unitary transformations of $S$ that commute with the $\Gamma$-action, and let $T$ be the scalar subgroup sitting inside $F$. 
 \begin{enumerate}
 \item Let the gauge group $\G^{F,\Gamma}$ of $E(\Gamma)$ be defined as $ \G^{F,\Gamma}= Map(S^2/\Gamma,F/T).$ Let $\g^{F,\Gamma}$ denote the Lie algebra of $\G^{F,\Gamma}$. We use $\rho$ to denote an element in $\G^{F,\Gamma}$, and we use $Y$ to denote an element in $\g^{F,\Gamma}$.
\item Let $\G^{F,\Gamma}_\C$ denote the complexification of $\G^{F,\Gamma}$, that is, $$\G^{F,\Gamma}_\C=Map(S^2/\Gamma,F^c/\C^*),$$ where $F^c=GL_\C(V)^\Gamma$ denotes the complex linear transformations of $S$ that commute with the $\Gamma$-action. We use $\kappa$ to denote an element in $\G^{F,\Gamma}_\C$.
\end{enumerate}
\end{defn}

\begin{defn}
We define the configuration space to be $\A^F \times C^\infty(S^2/\Gamma, E(\Gamma))$ where $\A^F$ and $C^\infty(S^2/\Gamma, E(\Gamma))$ are defined as follows.
\begin{enumerate}
\item Let $\A^F$ be the space of connections on $E(\Gamma)$ given by $$\A^F=\{A_0+\kappa^*\partial \kappa^{*-1}+\kappa^{-1}\bar{\partial}\kappa \mid  \kappa \in \G^{F,\Gamma}_\C \},$$ where $A_0$ is the aforementioned Chern connection on $H$ or equivalently $S^3$ thought of as the induced connection on $E(\Gamma)$. We will always denote $\kappa^*\partial \kappa^{*-1}+\kappa^{-1}\bar{\partial}\kappa$ by $B$, and sometimes we omit the base connection $A_0$.
\item Let $C^\infty(S^2/\Gamma, E(\Gamma))$ denote the abbreviation for the space of orbifold vector bundle sections $C^\infty_{orb}(S^2/\Gamma, E(\Gamma))$. 
\end{enumerate}

\end{defn}

\begin{rmk}
\begin{enumerate}
\item Notice that $\G^{F,\Gamma}$ is the subgroup of the group of unitary gauge automorphisms of $E(\Gamma)=E(V)^\Gamma=E(End(S))^\Gamma$ induced by the automorphisms of $E(S)^\Gamma$. And the action of $\rho \in \G^{F,\Gamma}$ on $\A^F \times C^\infty(S^2/\Gamma, E(\Gamma)) $ is given by the following: for a pair $(B,\Theta) \in \A^F \times C^\infty(S^2/\Gamma, E(\Gamma))$, $$\rho \cdot (B,\Theta)=(B+\rho d_{B}\rho^{-1},\rho \Theta \rho^{-1}).$$ Note that here we omit the base connection $A_0$ as $\rho$ fixes $A_0$. 
\item The action of the connection form $B$ on a section $\Theta$ comes from the representation of the Lie algebra of $F^c$ on $V$ induced from the representation $S$. 

\item The key point of the definition of $\A^F$ is that it can be thought of as the complex gauge orbit containing $A_0$, which will become important in the later sections.

\end{enumerate}
\end{rmk}

\begin{defn}[Symplectic structure on $\A^F \times C^\infty(S^2/\Gamma, E(\Gamma))$]
Let $(B_1,\Theta_1)$ and  $(B_2,\Theta)$ be in $\A^F \times C^\infty(S^2/\Gamma, E(\Gamma))$. Let a symplectic $2$-form $\boldsymbol{\Omega}$ on $\A^F \times C^\infty(S^2/\Gamma, E(\Gamma))$ be defined as follows: $$\boldsymbol{\Omega}((B_1,\Theta_1), (B_2,\Theta))=\int_{S^2/\Gamma}B_1\wedge B_2+\int_{S^2/\Gamma}-Im\langle\Theta_1,\Theta_2\rangle\omega_{vol}.$$ 
\end{defn}

\begin{defn}
\begin{enumerate}
\item Let $\G_0^{F,\Gamma}$ denote the based subgroup of $\G^{F,\Gamma}$, that is $$\G_0^{F,\Gamma}=\{\rho\in\G^{F,\Gamma} | \rho(x)=1,\text{for some fixed base point } x\in S^2/\Gamma \}.$$
We also get the complexified version $\G_{0,\C}^{F,\Gamma}$ for the above definition.
\item Let $\G_\tau^{F,\Gamma}$ denote the antipodal-invariant subgroup of $\G^{F,\Gamma}$ and let $\Omega^2_\tau(S^2/\Gamma;\f/\ttt)$ denote the antipodal-invariant subgroup of $\Omega^2(S^2/\Gamma;\f/\ttt)$, where $\tau:S^2 \to S^2$ is the antipodal map given by $x=(a,b,c) \mapsto \tau(x)=(-a,-b,-c)$. We can also think of $\tau$ as a map from $\C P^1$ to $\C P^1$ with $\tau: \C P^1 \to \C P^1, [z_1:z_2]\mapsto [-\bar{z}_2:\bar{z}_1]$. We remark here that $\tau$ commutes with the $\Gamma$-action and hence descends to a map $\tau: S^2/\Gamma \to S^2/\Gamma$.
\end{enumerate}
\end{defn}

Below, we define the $L^2$ inner product on various spaces.
\begin{defn}
\begin{enumerate}
\item Let $\Theta_1$, $\Theta_2$ be two sections of $E(\Gamma)$. We define the $L^2$ inner product of $\Theta_1$ and $\Theta_2$ to be $$\langle\Theta_1,\Theta_2\rangle_{L_2}=\int_{S^2/\Gamma}\langle \Theta_1,\Theta_2\rangle\omega_{vol}=\int_{S^2/\Gamma}\tr(\Theta_1\Theta_2^*)\omega_{vol},$$ where $\langle \Theta_1,\Theta_2\rangle_x=h_{E(V)}(\Theta_1(x),\Theta_2(x))_x=\tr(\Theta_1(x)\Theta_2^*(x))_x$.
\item We identify $\Omega^0(S^2/\Gamma;\f/\ttt)$ and $\Omega^2(S^2/\Gamma;\f/\ttt)$ as dual spaces through the following integration: let $\phi_1\in\Omega^0(S^2/\Gamma;\f/\ttt)$ and $\phi_2\in\Omega^2(S^2/\Gamma;\f/\ttt)$, then $\phi_2(\phi_1)=\int_{S^2/\Gamma}\langle\phi_1,\phi_2\rangle$, where we think of $\phi_2$ as an element in $\Omega^0(S^2/\Gamma;\f/\ttt)$ multiplied by the volume form $\omega_{vol}$, and the inner product is pointwisely given by the inner product on $\f/\ttt$.

\end{enumerate}
\end{defn}

\section{An Overview of the Gauge-Theoretic Construction}

In this section, we describe the main gauge-theoretic construction of the ALE spaces while leaving some details of the construction and most proofs to the following sections. We make an important remark that from this point on and throughout the rest of the paper, we take the $\Gamma$-representation $S$ to be the regular representation $R$ of $\Gamma$ unless otherwise specified, and carry on with the same notations introduced in the previous sections. In particular, we have $V=End(R)$. 
\subsection{Symplectic reduction}
Recall that in the previous section, we define the gauge group to be $ \G^{F,\Gamma}= Map(S^2/\Gamma,F/T)$ acting on the configuration space $\A^F \times C^\infty(S^2/\Gamma, E(\Gamma))$ under the following action: for $\rho \in \G^{F,\Gamma}$, and $(B,\Theta)\in\A^F \times C^\infty(S^2/\Gamma, E(\Gamma))$,
$$\rho\cdot(B,\Theta)=(B+\rho d_{B}\rho^{-1},\rho \Theta \rho^{-1}).$$

\begin{prop}\label{sympred}
The above gauge group action on  $\A^F \times C^\infty(S^2/\Gamma, E(\Gamma))$ is Hamiltonian and gives rise to the following moment map:
$$\tilde{\mu}_1: \A^F \times C^\infty(S^2/\Gamma, E(\Gamma)) \to \Omega^2(S^2/\Gamma;\f/\ttt),$$$$ (B,\Theta)\mapsto F_B-\frac{i}{2}[\Theta,\Theta^*]\omega_{vol}.$$

\end{prop}

\begin{rmk}
\begin{enumerate}
\item
Notice that $B$ alone isn't a connection whereas both $A_0$ and $A_0+B$ are connections on $E(\Gamma)$. Hence, we can write $F_{A_0+B}=F_{A_0}+F_B$, and $\bar{\partial}_{A_0+B}=\bar{\partial}_{A_0}+B^{0,1}$, where $F_B=F_{A_0+B}-F_{A_0}$. More explicitly, $F_B=dB+B\wedge B$.
\item With the preceding proposition in place, we can write down the following equations: for $\tilde{\zeta}_1\in Z$, where $Z$ is the center of $(\f/\ttt)^*$ thought of as traceless matrices in $\f/\ttt$, we consider
\begin{equation} \label{holo}
\bar{\partial}_{A_0+B} \Theta=0
\end{equation}

\begin{equation} \label{eq1}
F_B-\frac{i}{2}[\Theta,\Theta^*]\omega_{vol}=\tilde{\zeta}_1\cdot\omega_{vol}
\end{equation}
\end{enumerate}
\end{rmk}

The above equations motivate the following definition.
\begin{defn}

 For an element $\tilde{\zeta}_1\in Z$, let $\M(\Gamma, \tilde{\zeta}_1)$ be the moduli space of solutions to \ref{holo} and \ref{eq1} that lie in the configuration space $\A^F \times {C}^\infty(S^2/\Gamma, E(\Gamma))$ modulo the gauge group action, that is, $$\M(\Gamma, \tilde{\zeta}_1)=\{(A_0+B,\Theta) \in\A^F \times {C}^\infty(S^2/\Gamma, E(\Gamma))\mid  (\ref{holo})-(\ref{eq1})\}/\G^{F,\Gamma}.$$

\end{defn}

\begin{prop}\label{moduli1}
For choices of $\tilde{\zeta}_1$ such that $\G^{F,\Gamma}$ acts freely on the space of solutions to \ref{holo} and \ref{eq1} in $\A^F \times C^\infty(S^2/\Gamma,E(\Gamma))$, $\M(\Gamma, \tilde{\zeta}_1)$  can be identified with $\mu_1(\tilde{\zeta}_1)^{-1}/F$ in \cite{kronheimer}.
\end{prop}

\begin{rmk}
We will discuss the conditions assumed in the above proposition in detail in the following sections and we will prove the proposition in Section 7. 
\end{rmk}

\subsection{Further reduction}

Everything regarding to the hyperk{\"a}hler structure on $C^\infty(S^2/\Gamma, E(\Gamma))$ appearing in this subsection will be discussed in detail in Section 4. Here we give a brief overview. It turns out that $C^\infty(S^2/\Gamma, E(\Gamma))$ can be given a hyperk{\"a}hler structure. 

Before we write down the K{\"a}hler forms, we first introduce some notations. For a section $\Theta$ of $C^\infty(S^2/\Gamma, E(\Gamma))$, we identify $\Theta$ with an $S^1$- and $\Gamma$-equivariant map $\lambda: S^3 \to End(R)$, and hence we can express $\Theta$ as $\Theta: x \mapsto [p, \lambda(p)]$, for $x \in S^2$ and $p \in \pi^{-1}(x) \subset S^3$.

There is a complex structure $J$, in addition to the standard complex structure $I$, on the space of sections $C^\infty(S^2/\Gamma, E(\Gamma))$, which we can express as follows. Let $\Theta: x \mapsto [p, \lambda(p)]$ be given, the action of $J$ on $\Theta$ is given by $J\Theta: x \mapsto [p, -\lambda(J(p))^*]$, where $p \in S^3$ and $J$ on $S^3$ is just the usual quaternion action. 

\begin{prop}\label{hkmetric}
There are three symplectic forms on $C^\infty(S^2/\Gamma, E(\Gamma))$ compatible with complex structures $I$, $J$, $K$, respectively:

$$\omega_1(\Theta_1,\Theta_2)=\int_{S^2/\Gamma}-Im\langle \Theta_1,\Theta_2\rangle\omega_{vol},$$
$$\omega_2(\Theta_1,\Theta_2)=\int_{S^2/\Gamma}Re \langle J\Theta_1, \Theta_2\rangle \omega_{vol},$$
$$\omega_3(\Theta_1,\Theta_2)=\int_{S^2/\Gamma}-Im\langle J\Theta_1, \Theta_2\rangle\omega_{vol},$$
and a hyperk{\"a}hler metric $g_h$ such that
$$g_h(\Theta_1,\Theta_2)=\int_{S^2/\Gamma}Re\langle\Theta_1,\Theta_2\rangle\omega_{vol},$$
together giving rise to a hyperk{\"a}hler structure on $C^\infty(S^2/\Gamma, E(\Gamma))$.

\end{prop}

We will prove the above proposition in Section 4. It turns out that the action of the $\tau$-invariant gauge group $\G^{F,\Gamma}_\tau$ on the space of sections of $E(\Gamma)$ with respect to each one of the three symplectic forms is again Hamiltonian. Hence, we can write down the following additional moment maps and operate a further reduction on the configuration space:

\begin{itemize}

\item $\tilde{\mu}_2:  C^\infty(S^2/\Gamma, E(\Gamma))\to \Omega^2(S^2/\Gamma;\f/\ttt), \Theta \mapsto -\frac{1}{4}([J\Theta,\Theta^*]-[\Theta,J\Theta^*])\omega_{vol}$,

\item $\tilde{\mu}_3:   C^\infty(S^2/\Gamma, E(\Gamma))\to \Omega^2(S^2/\Gamma;\f/\ttt), \Theta \mapsto -\frac{i}{4}([J\Theta,\Theta^*]+[\Theta,J\Theta^*])\omega_{vol}$.

\end{itemize}

We get two additional moment map equations: let $\tilde{\zeta}_2, \tilde{\zeta}_3 \in Z$, consider

\begin{equation}\label{mu2}
-\frac{1}{4}([J\Theta,\Theta^*]-[\Theta,J\Theta^*])\omega_{vol}=\tilde{\zeta}_2\cdot\omega_{vol},
\end{equation}

\begin{equation}\label{mu3}
-\frac{i}{4}([J\Theta,\Theta^*]+[\Theta,J\Theta^*])\omega_{vol}=\tilde{\zeta}_3\cdot\omega_{vol}.
\end{equation}

\begin{defn}
Let $\A^F_\tau\subset \A^F$ be the subspace of connections in $\A^F$ on $E(\Gamma)$ given by $$\A^F_\tau=\{A_0+\kappa^*\partial \kappa^{*-1}+\kappa^{-1}\bar{\partial}\kappa \mid  \kappa \in \G^{F,\Gamma}_{\tau,\C} \},$$ where $A_0$ is again the base Chern connection on $E(\Gamma)$, and  $\G^{F,\Gamma}_{\tau,\C}$ is the complexification of the $\tau$-invariant subgroup $\G^{F,\Gamma}_\tau$. 
\end{defn}

\begin{rmk}
\begin{enumerate}
\item
We will make the statement of ``a suitable choice of $\tilde{\zeta}$" in Theorem \ref{main} precise in Section $7$ where we also prove the theorem.  

\item We remark that \ref{holo} -- \ref{eq1} resemble Hitchin's equations but they are not the same. First of all, the bundle considered here with respect to which these equations are written is different from that of Hitchin; furthermore, the pair consists of a connection and a section whereas Hitchin's equations consider a connection and a $(1,0)$-form. 
\end{enumerate}
\end{rmk}

\subsection{Proof of Proposition \ref{sympred}}
Here in this subsection, we give the proof of Proposition \ref{sympred}, which involves simply standard calculations.

\begin{proof}[Proof of Proposition \ref{sympred}]
We will show that $F_B-\frac{i}{2}[\Theta,\Theta^*]\omega_{vol}$ is a moment map on $\Omega^1(S^2/\Gamma; \f/\ttt)\times C^\infty(S^2/\Gamma, E(\Gamma))$ induced by the action of $\G^{F,\Gamma}$. We need to check the two properties of a moment map.

Let  $Y: S^2/\Gamma \to \f/\ttt$ be in $\g^{F,\Gamma}$, and let $Y^\sharp$ be the vector field on $\Omega^1(S^2/\Gamma; \f/\ttt)\times C^\infty(S^2/\Gamma, E(\Gamma))$ generated by $Y$. Then $Y^\sharp(B, \Theta)$ is given by $$\frac{d}{dt}|_{t=0}(B+\exp(tY)d_B\exp(-tY),\exp(tY)\Theta \exp(-tY))=(-d_B Y,[Y,\Theta]).$$

Hence, we have $$\iota_{Y^\sharp}\omega_{(B,\Theta)}(B',\Theta')=$$$$\int_{S^2/\Gamma}
\tr(-d_B Y\wedge B')-\int_{S^2/\Gamma}Im\langle[Y,\Theta],\Theta'\rangle\omega_{vol}=$$$$\int_{S^2/\Gamma}
\tr([Y,B]\wedge B'-dY\wedge B')-\int_{S^2/\Gamma}Im\langle [Y,\Theta],\Theta'\rangle \omega_{vol}.$$ Meanwhile, let $(B_t,\Theta_t)_{t\in[0,1]}$ be a path in $\Omega^1(S^2/\Gamma; \f/\ttt)\times C^\infty(S^2/\Gamma, E(\Gamma))$ such that $(B_0,\Theta_0)=(B,\Theta)$ and $\frac{d}{dt}|_{t=0}(B_t,\Theta_t)=(B',\Theta')$. Then we also have $$d\tilde{\mu}^Y_{1(B,\Theta)}(B',\Theta')=$$$$\frac{d}{dt}|_{t=0}\int_{S^2/\Gamma}\tr( Y\wedge F_{B_t})-\frac{d}{dt}|_{t=0}\int_{S^2/\Gamma}\langle Y,\frac{i}{2} [\Theta_t,\Theta_t^*]\rangle\omega_{vol}=$$$$\frac{d}{dt}|_{t=0}\int_{S^2/\Gamma}\tr( Y\wedge F_{B_t})-\frac{d}{dt}|_{t=0}\int_{S^2/\Gamma}-\frac{i}{2}\langle Y, [\Theta_t,\Theta_t^*]\rangle\omega_{vol}=$$$$\int_{S^2/\Gamma}\tr( Y\wedge (dB'+B'\wedge B+B\wedge B'))-\int_{S^2/\Gamma}-\frac{i}{2}\langle Y, [\Theta',\Theta^*]+[\Theta,\Theta'^{*}]\rangle\omega_{vol}.$$
Hence, we have $\iota_{Y^\sharp}\omega_{(B,\Theta)}(B',\Theta')=d\tilde{\mu}^{Y}_{1(B,\Theta)}(B',\Theta')$. 

We also need to check the equivariance condition, that is, $\tilde{\mu}_1\circ \psi_\rho=Ad^*_\rho \circ \tilde{\mu}_1$. Let  $\rho$ be an element in the unitary gauge group $\G^{F,\Gamma}$, and let  $$\psi_\rho: \Omega^1(S^2/\Gamma; \f/\ttt)\times C^\infty(S^2/\Gamma, E(\Gamma)) \to \Omega^1(S^2/\Gamma; \f/\ttt)\times C^\infty(S^2/\Gamma, E(\Gamma))$$ be the diffeomorphism on the configuration space induced by $\rho$. We have $$\tilde{\mu}_1\circ \psi_\rho (B,\Theta)=F(B+\rho d_B \rho^{-1})-\frac{i}{2}[\rho\Theta\rho^{-1},(\rho\Theta\rho^{-1})^*]\omega_{vol}.$$ Meanwhile, $$Ad^*_\rho \circ \tilde{\mu}_1(B,\Theta)=\rho F_B \rho^{-1}-\frac{i}{2}\rho[\Theta,\Theta^*]\rho^{-1}\omega_{vol}.$$ Since the gauge action on curvature is conjugation and $\rho^{-1}=\rho^*$, we have the desired equality $$\tilde{\mu}_1\circ \psi_\rho(B,\Theta)=Ad^*_\rho \circ \tilde{\mu}_1(B,\Theta).$$ 
\end{proof}

\section{Hyperk{\"a}hler structure on $C^\infty(S^2/\Gamma,E(\Gamma))$}

\subsection{The quaternions}
Let $SU(2)$ denote the 2-dimensional special unitary group. Explicitly,  $SU(2)=\{\gamma\in Gl_2(\C)| \gamma=\begin{pmatrix}
u & v \\
-v^* & u^* 
\end{pmatrix}, |u|^2+|v|^2=1\}$. Note that $u^*$ refers to complex conjugation.

Below, we will set up another piece of convention, that is, to endow $\C^2$ with a right $\HH$-module structure.  Write $(z_1,z_2)$ for a point in $\C^2$, where $I,J,K$ act on $\C^2$ as follows: $I(z_1,z_2)=(iz_1,iz_2)$, $J(z_1,z_2)=(-z_2^*,z_1^*)$, $K(z_1,z_2)=(-iz_2^*,iz_1^*)$. 

Note we also have $SU(2)$ acting on the right on $\C^2$: Let $\gamma \in SU(2)$, $\gamma=
\begin{pmatrix}
u & v \\
-v^* & u^* 
\end{pmatrix}$,
then $$J((z_1,z_2)\gamma)=J(uz_1-v^*z_2,vz_1+u^*z_2)=(-v^*z_1^*-uz_2^*,u^*z_1^*-vz_2^*),$$ and $$(J(z_1,z_2))\gamma=(-z_2^*,z_1^*)\gamma=(-v^*z_1^*-uz_2^*,u^*z_1^*-vz_2^*),$$ so the $SU(2)$-action commutes with the $J$-action. Hence, the $SU(2)$-action on $\C^2$ commutes with all the $I$-, $J$-, $K$-actions.

If we restrict the actions to $S^3$ thought of as the unit quaternions, then we make the following observations:

\begin{lem}
The $S^1$-action on $S^3$ coming from the dual Hopf fibration commutes with $I$, and for $p \in S^3$, $g\in S^1$, $J(pg)=J(p)g^*$, $K(pg)=K(p)g^*$.
\end{lem}

\begin{lem}
Consider $S^3$ as the principal $S^1$-bundle via the dual Hopf fibration. Let $\pi: S^3 \to \C P^1$ be the projection map where $\pi(z_1,z_2)=[z_1:z_2]$. Then $I$ acts as the identity and $J$, $K$ act as the natural involution on the base $\C P^1$ given by $\tau: \C P^1\to \C P^1$, $[z_1:z_2]\mapsto[-z_2^*:z_1^*]$.
\end{lem}

\subsection{Quaternionic structures on associated bundles and spaces of sections}
We have previously introduced the bundle $E(V)$ and $E(\Gamma)$. In this subsection, we introduce quaternionic structures on these bundles and their spaces of sections. 

We begin with $E(V)$. Notice that we can define the quaternion actions on $E(V)$ in the following way: $$I[p,v]=[-I(p),v]=[p,iv],$$ $$J[p,v]=[J(p),v^*],$$ $$K[p,v]=[-K(p),v^*],$$ with $[p,v] \in E(V).$ 

It's straightforward to check that the $I$-, $J$-, $K$-actions defined above satisfy the properties for quaternion actions. Hence, we have equipped $E(V)$ with a quaternionic structure.

Now, we move on to $E(\Gamma)$. In the previous subsection, we have shown that the $\Gamma$-action and the $J$-action commute on $\C^2$. Observe that we have that the $\Gamma$-action commutes with the quaternion actions on the level of $E(V)$ as well; more precisely, we have that $$J(\gamma[p,v])=J[ p\gamma,\gamma^{-1} v]=J[ p\gamma, R(\gamma^{-1})vR(\gamma)]$$$$=[J(p\gamma), (R(\gamma^{-1})vR(\gamma))^*]=[J(p)\gamma, R(\gamma)^*v^*R(\gamma^*)^*]$$$$=[J(p)\gamma, R(\gamma^*)v^*R(\gamma)]=[J(p)\gamma, R(\gamma^{-1})v^*R(\gamma)]=\gamma(J[p,v]),$$ given that $\gamma \in SU(2)$ and $R:\Gamma \to U(R)\subset End(R)$ is the regular representation. 

Hence, the quaternion actions descend to $E(\Gamma)$. We remark that the $J$- and $K$-actions on $E(V)$ and $E(\Gamma)$ act on the base by $\tau$ which we have introduced previously.

\begin{prop}
The map $I: E(V) \to E(V)$ is an isometry with respect to the hermitian metric on $E(V)$, and $J, K: E(V) \to E(V)$ are skew-isometries in the sense that $\langle J[p,v_1],J[p,v_2]\rangle=\overline{\langle v_1,v_2\rangle}$, $\langle K[p,v_1],K[p,v_2]\rangle=\overline{\langle v_1,v_2\rangle}$, for $[p,v_1],[p,v_2]\in E(V)_x$, for all $x\in S^2$.
\end{prop}

From here, by pullbacks, we can make the spaces of sections $C^\infty(E(V))$ and $C^\infty(E(\Gamma))$ into right $\HH$-modules. We will focus on $C^\infty(E(\Gamma))$ here but the statements for $C^\infty(E(V))$ are exactly the same. 

\begin{prop}\label{hkstruc}
The space of sections $C^\infty(E(\Gamma))$ of $E(\Gamma)$ is an infinite-dimensional right $\HH$-module with the following quaternion actions: for $\Theta$ a section of $E(\Gamma)$, we identify $\Theta$ with a map $\lambda:S^3\to End(R)$ equivariant with respect to the $S^1$- and $\Gamma$-action, and we define that for $\Theta: x \mapsto [p,\lambda(p)]$, $$I\Theta: x \mapsto [p,i\lambda(p)],$$ $$J\Theta: x \mapsto [p,-\lambda(J(p))^*],$$ $$K\Theta: x \mapsto [p,\lambda(K(p))^*],$$
 where $J(p)$ and $K(p)$ are the usual $J$-, $K$- actions on $S^3$. 
 \end{prop}

We leave out the proofs for the above propositions as they involve simply using and checking the properties of quaternion actions. Also, Proposition \ref{hkstruc} holds for the space of sections $C^\infty(E(V))$ of $E(V)$ with appropriate modifications of adjectives. 

\subsection{Hyperk{\"a}hler structure on the space of sections $C^\infty(E(\Gamma))$}
With the previous observations involving the quaternion actions, we are now ready to introduce the hyperk{\"a}hler structure on the space of sections $C^\infty(E(\Gamma))$ that will be relevant to the construction. We cite \cite{hitchin-supersymm} for the inspiration of this subsection. We remark that the same analysis below will give rise to hyperk{\"a}hler structures of $C^\infty(E(V))$ as well; in fact, we can even replace the regular representation $R$ with any complex $\Gamma$-representation $S$ and obtain a hyperk{\"a}hler structure on $C^\infty(E(End(S)))$, as we use no specific properties of the regular representation $R$ for defining the hyperk{\"a}hler structure.

Recall that in the previous subsection, we have that for $\Theta: x \mapsto [p,\lambda(p)]$, the action of $J$ on $\Theta$ is such that $$J\Theta: x \mapsto [p,-\lambda(J(p))^*],$$ where $J(p)$ is the usual $J$-action on $S^3$. 

We now give the proof of Proposition \ref{hkmetric}. 

\begin{proof}[Proof of Proposition \ref{hkmetric}]
We focus on $\omega_3$. First we make the observation that   $$\omega_3(\Theta_1,\Theta_2)=\int_{S^2/\Gamma}-Im\langle J\Theta_1, \Theta_2\rangle\omega_{vol}=\int_{S^2/\Gamma}-Im\langle -J\Theta_2,\Theta_1\rangle\omega_{vol}$$$$=\int_{S^2/\Gamma}-Im\langle J\Theta_2,\Theta_1\rangle\tau^*\omega_{vol}=-\omega_3(\Theta_2,\Theta_1).$$ Indeed, for $\Theta_1: x \mapsto [p, \lambda_1(p)]$ and $\Theta_2: x \mapsto [p, \lambda_2(p)]$, we have $$\langle J\Theta_1, \Theta_2\rangle_x=\tr(-\lambda_1(J(p))^*\lambda_2(p)^*)$$ and $$\langle -J\Theta_2,\Theta_1\rangle_x=\tr(\lambda_2(J(p))^*\lambda_1(p)^*)=\tr(\lambda_1(p)^*\lambda_2(J(p))^*).$$Since $J$ acts on $S^2/\Gamma$ by $\tau$ which has the property that $\tau^*\omega_{vol}=-\omega_{vol}$, we have the desired equality after integration. This gives $\omega_3$ the skew-symmetric property of a symplectic form. The same can be shown for $\omega_2$. The properties of $\omega_2$ and $\omega_3$ being closed and non-degenerate are obvious. We hence can also write down the compatible hyperk{\"a}hler metric $g_h$ on $C^\infty(E(\Gamma))$: $$g_h(\Theta_1,\Theta_2)=\int_{S^2/\Gamma}Re\langle\Theta_1,\Theta_2\rangle\omega_{vol},$$ and it is evident that $g_h$ is compatible with the complex structures and the symplectic forms.

\end{proof}

Next, we want to justify the two additional moment map equations, \ref{mu2} and \ref{mu3}. To start with, we make the observation that for $\Theta: x \mapsto [p, \lambda(p)]$ and $Y:S^2/\Gamma \to \f/\ttt$ an element in $\g^{F,\Gamma}$, we have $$Y\Theta-\Theta Y: x \mapsto [p,Y(x)\lambda(p)-\lambda(p)Y(x)]$$ and $$J(Y\Theta-\Theta Y): x\mapsto [p,-\lambda(J(p))^*Y(\tau(x))^*+Y(\tau(x))^*\lambda(J(p))^*].$$Thus, we can think of $J(Y\Theta-\Theta Y)=[J\Theta,(\tau^*Y)^*]$, where $\tau$ denotes the involution we have introduced previously. Meanwhile, for $YJ\Theta-J\Theta Y$, we have $$YJ\Theta-J\Theta Y: x \mapsto [p, -Y(x)\lambda(J(p))^*+\lambda(J(p))^*Y(x)].$$ Hence, for $Y: S^2/\Gamma \to \f/\ttt$ invariant under $\tau$, that is, $Y(x)=Y(\tau(x)), \forall x \in S^2/\Gamma$, we have 
\begin{equation}\label{tau}
J(Y\Theta-\Theta Y)=[J\Theta,(\tau^*Y)^*]=[J\Theta,-Y]=[Y,J\Theta]=YJ\Theta-J\Theta Y.
\end{equation}

\begin{prop}\label{khmomap}
The action of the $\tau$-invariant subgroup $\G^{F,\Gamma}_\tau$ of $\G^{F,\Gamma}$  on $C^\infty(E(\Gamma))$ is Hamiltonian with respect to the symplectic forms $\omega_2$ and $\omega_3$ and gives rise to the following moment maps: $$\tilde{\mu}_2:   C^\infty(S^2/\Gamma, E(\Gamma))\to \Omega^2(S^2/\Gamma;\f/\ttt),$$ $$\Theta \mapsto -\frac{1}{4}([J\Theta,\Theta^*]-[\Theta,J\Theta^*])\omega_{vol},$$ and $$\tilde{\mu}_3:   C^\infty(S^2/\Gamma, E(\Gamma))\to \Omega^2(S^2/\Gamma;\f/\ttt),$$ $$\Theta \mapsto -\frac{i}{4}([J\Theta,\Theta^*]+[\Theta,J\Theta^*])\omega_{vol}.$$

\end{prop}

\begin{proof}
Again, we first focus on $\omega_3$. Similar to the proof of Proposition 3.1, we let $Y: S^2/\Gamma \to \f/\ttt$ be a $\tau$-invariant element in $\g^{F,\Gamma}$ and let $Y^\sharp$ denote the vector field  on $C^\infty(E(\Gamma))$ induced by $Y$. 

Now, let's compute $\iota_{Y^\sharp}{\omega_3}_\Theta(\Theta')$. We have $$\iota_{Y^\sharp}{\omega_3}_\Theta(\Theta')=\int_{S^2/\Gamma}-Im\langle J[Y,\Theta], \Theta'\rangle\omega_{vol}=\int_{S^2/\Gamma}-Im\langle J(Y\Theta-\Theta Y), \Theta'\rangle\omega_{vol}.$$ Hence,  by \ref{tau}, we have
$$\int_{S^2/\Gamma}-Im\langle J(Y\Theta-\Theta Y), \Theta'\rangle\omega_{vol}=$$$$\int_{S^2/\Gamma}-Im\langle [J\Theta,Y^*], \Theta'\rangle\omega_{vol}=\int_{S^2/\Gamma}Im\tr([Y^*,J\Theta] \Theta'^{*})\omega_{vol}$$$$=\int_{S^2/\Gamma}\frac{i}{2}\tr([\Theta',J\Theta^*]Y^*+[J\Theta,\Theta'^{*}]Y^*)\omega_{vol}$$$$=\int_{S^2/\Gamma}\frac{i}{2}(\langle[\Theta',J\Theta^*],Y\rangle+\langle[J\Theta,\Theta'^{*}],Y\rangle)\omega_{vol}.$$

Meanwhile, by the skew-symmetric property of $\omega_3$, we also have 
$$\int_{S^2/\Gamma}-Im\langle J[Y,\Theta], \Theta'\rangle\omega_{vol}=\int_{S^2/\Gamma}-Im\langle -J\Theta', [Y,\Theta]\rangle\omega_{vol}$$
$$=\int_{S^2/\Gamma}Im \tr([Y,\Theta^*]J\Theta')\omega_{vol}=\int_{S^2/\Gamma}\frac{i}{2}\tr([J\Theta'^*,\Theta]Y+[\Theta^*,J\Theta']Y)\omega_{vol}$$$$=\int_{S^2/\Gamma}\frac{i}{2}(\langle[\Theta,J\Theta'^*],Y\rangle+\langle [J\Theta',\Theta^*],Y\rangle)\omega_{vol}.$$

Now, we obtain the following:
$$2\iota_{Y^\sharp}{\omega_3}_\Theta(\Theta')=$$$$\int_{S^2/\Gamma}\frac{i}{2}\langle[\Theta',J\Theta^*]+[\Theta,J\Theta'^*],Y\rangle\omega_{vol}+\int_{S^2/\Gamma}\frac{i}{2}\langle [J\Theta,\Theta'^*]+[J\Theta',\Theta^*],Y\rangle\omega_{vol}.$$

On the other hand, let $\Theta_t$ with $t\in[0,1]$ be a path in $ C^\infty(S^2/\Gamma, E(\Gamma))$ such that $\Theta_0=\Theta$ and $\frac{d}{dt}|_{t=0}\Theta_t=\Theta'$. Then we have $$d\tilde{\mu}^Y_{3_\Theta}(\Theta')=\frac{d}{dt}|_{t=0}\int_{S^2/\Gamma}-\langle Y,\frac{i}{4}([J\Theta_t,\Theta_t^*]+[\Theta_t,J\Theta_t^*])\rangle\omega_{vol}$$$$=\int_{S^2/\Gamma}-\langle Y,\frac{i}{4}([J\Theta',\Theta^*]+[J\Theta,\Theta'^*])\rangle\omega_{vol}+\int_{S^2/\Gamma}-\langle Y,\frac{i}{4}([\Theta',J\Theta^*]+[\Theta,J\Theta'^*])\rangle\omega_{vol}.$$

The above computations verify that $$ \tilde{\mu}_3(\Theta)=-\frac{i}{4}([J\Theta,\Theta^*]+[\Theta,J\Theta^*])\omega_{vol}.$$ 

By very similar computations, we also get that for $$\omega_2(\Theta_1,\Theta_2)=\int_{S^2/\Gamma}Re \langle J\Theta_1, \Theta_2\rangle \omega_{vol},$$ we have $$\tilde{\mu}_2(\Theta)=-\frac{1}{4}([J\Theta,\Theta^*]-[\Theta,J\Theta^*])\omega_{vol}.$$

We leave out the proof for the equivariance condition as it is essentially the same as that of Proposition \ref{sympred}.

\end{proof}

\begin{rmk}
\begin{enumerate}
\item Note, here we need to restrict the gauge group action to the $\tau$-invariant subgroup $\G^{F,\Gamma}_\tau$ which is different from the previous setup.

\item
Observe that $\frac{i}{4}([J\Theta,\Theta^*]+[\Theta,J\Theta^*])$ and $\frac{1}{4}([J\Theta,\Theta^*]-[\Theta,J\Theta^*])$ are both $\tau$-invariant and hence the new moment maps map into the correct space. 
\end{enumerate}
\end{rmk}

\begin{lem}\label{matrixid}
If $\Theta$ is holomorphic with respect to a fixed holomorphic structure on $E(\Gamma)$ and is identified with a pair of matrices $(\alpha,\beta)$, then $J\Theta=J(\alpha,\beta)=(-\beta^*,\alpha^*)$.
\end{lem}

\begin{proof}
As before, we express $\Theta$ as $\Theta: x\mapsto [p,\lambda(p)]$, where $\lambda: S^3\to End(R)$ is $S^1$- and $\Gamma$-equivariant. Since $\Theta$ is holomorphic, $\lambda$ can be extended to a complex linear map $\lambda: \C^2\to End(R)$. Hence, $\lambda$ can be thought of as a pair of matrices $(\alpha,\beta)$ such that $\lambda(z_1,z_2)=z_1\alpha+z_2\beta$, for $(z_1,z_2)\in \C^2$. 

On the other hand, we have $J\Theta: x\mapsto [p,-\lambda(J(p))^*]$. This give us $$-\lambda(J(z_1,z_2))^*=-\lambda(-z_2^*,z_1^*)^*=-(-z_2^*\alpha+z_1^*\beta)^*=-z_1\beta^*+z_2\alpha^*.$$ This precisely says that $J\Theta$ reduces to $(-\beta^*,\alpha^*)$.
\end{proof}

\begin{rmk}
\begin{enumerate}
\item
Provided with the previous lemma, we observe that for $\tilde{\mu}_3$, we have $$\frac{i}{4}([J\Theta,\Theta^*]+[\Theta,J\Theta^*])$$$$=\frac{i}{4}([-\beta^*,\alpha^*]+[\alpha^*,\beta^*]+[\alpha,-\beta]+[\beta,\alpha])$$$$=\frac{i}{4}(2[\alpha^*,\beta^*]-2[\alpha,\beta])=\frac{i}{2}([\alpha^*,\beta^*]-[\alpha,\beta]),$$ but this is precisely the third moment map $\mu_3$ in Kronheimer's setup \cite{kronheimer}; similar calculations show that $\tilde{\mu}_2$ also reduces to $\mu_2$ in Kronheimer's setup \cite{kronheimer}. This observation will become a key element in the proof of Theorem 3.6. 

\item We remark that the same analysis presented in this section will give rise to hyperk{\"a}hler structures to $C^\infty(E(End(S)))$ and $C^\infty(E(End(S))_r^\Gamma)$ if we replace the regular representation $R$ with any $\Gamma$-representation $r$ on $S$ with an appropriately chosen hermitian structure to obtain a hyperk{\"a}hler structure on $C^\infty(E(End(S)))$ and $C^\infty(E(End(S))_r^\Gamma)$ , as we use no specific properties of the regular representation $R$ for defining the hyperk{\"a}hler structure. 

\end{enumerate}
\end{rmk}

\section{Uniqueness theorems}

In this section, we analyze both the unitary gauge group action and the complex gauge group action on the configuration space $\A^F\times C^\infty(E(\Gamma))$. In particular, we prove two uniqueness theorems: the first one states that any solution to \ref{holo} and \ref{eq1} lying in $\A^F\times C^\infty(E(\Gamma))$ that are $\G^{F,\Gamma}_\C$-equivalent are also $\G^{F,\Gamma}$-equivalent, which is a standard occurrence in gauge theory. The second uniqueness theorem can be thought of as a corollary of the first one, which states that any solution to \ref{holo} -- \ref{mu3} lying in $\A_\tau^F\times C^\infty(E(\Gamma))$ that are $\G^{F,\Gamma}_{\tau,\C}$-equivalent must also be $\G^{F,\Gamma}_\tau$-equivalent.

\begin{lem}\label{transaction}
Up to automorphisms of $E(\Gamma)$, the space $\A^F$ defines a single holomorphic structure on $E(\Gamma)$,  identifying $E(\Gamma)$ with the direct sum of hyperplane bundles holomorphically. 
\end{lem}

\begin{proof}

By construction, $A_0$ is taken to be the Chern connection giving rise to the holomorphic structure on $E(\Gamma)$ such that $E(\Gamma)$ splits holomorphically as a direct sum of hyperplane bundles. As $\A^F$ is simply defined to be the complex orbit containing $A_0$, we must have that $\A^F$ defines a single holomorphic structure identifying $E(\Gamma)$ with the direct sum of hyperplane bundles holomorphically, as stated in the lemma.

\end{proof}

\begin{lem}\label{transaction2}
The based complex gauge group acts freely on $\A^F$, and the stabilizer of $B$ in the complex gauge group is isomorphic to the constant subgroup. 
\end{lem}

\begin{rmk} The two preceding lemmas can both be formulated where we replace $\A^F$ with $\A^F_\tau$ and use the corresponding $\tau$-invariant gauge groups.
\end{rmk}

\begin{defn} Let $Q$ be the canonical $2$-dimensional representation of $SU(2)$. Let $Hom(Q,V)^{\Gamma}$ denote the $\Gamma$-invariant subset of $Hom(Q,V)$, consisting of all maps that commute with the $\Gamma$-actions on $Q$ and $V$, that is, for $f \in Hom(Q,V)$,$f(\gamma(z))=\gamma(f(z))$, where $\gamma \in \Gamma$ and $z \in Q$. 
\end{defn}

\begin{lem}\label{holosecid}
The space $Hom(Q,V)$ is isomorphic to the space of holomorphic sections of $E(V)$ with respect to $A_0$. The space $Hom(Q,V)^\Gamma$ is isomorphic to the space of holomorphic sections of $E(\Gamma)$ with respect to $A_0$. 
\end{lem}

\begin{rmk}
\begin{enumerate}
\item It is easy to see that $M \cong Hom(Q, V )^\Gamma$, and hence by the previous lemma, we can think of $M$ as the space of holomorphic sections of $E(\Gamma)$ with respect to the fixed connection $A_0$. 
\item The above lemma gives rise to a map $$\Psi: M \to \A^F \times C^\infty(E(\Gamma))$$$$\lambda \mapsto (A_0,\Theta: x\mapsto[p,\lambda(p)]),$$ with the property such that $\Psi$ is an isomorphism onto its image.  In addition, $\Psi$ can be naturally regarded as an isometry onto its image. To see this, we observe that the hyperk{\"a}hler metric $g_h$ given in Proposition 3.6 restricted to the set $\{\Theta\in C^{\infty}(E(\Gamma))| \bar{\partial}_{A_0}\Theta =0\}$ agrees with the natural flat hyperk{\"a}hler metric on $M$. Hence, $\Psi$ is an isometry onto its image.

\item A holomorphic section of $E(\Gamma)$ with respect to the fixed connection $A_0$ can be expressed as a pair of matrices $(\alpha,\beta)$ where $(\alpha,\beta)$ is $\Gamma$-invariant as in \cite{kronheimer}.

\end{enumerate}
\end{rmk}

We omit the proofs for the two preceding lemmas as the proofs can be found in or follow from standard references such as  \cite{kobayashi2}, \cite{kobayashi1} and \cite{griffiths-harris}.

\begin{lem}\label{psimap}
There is a map $$\tilde{\Psi}: M\to \{(A_0+B,\Theta)\in \A^F\times C^{\infty}(E(\Gamma))|\bar{\partial}_{A_0+B}\Theta=0\}/\G^{F,\Gamma}_{0,\C} $$ such that $\tilde{\Psi}$ is an isomorphism, where $M$ comes from Kronheimer's construction in \cite{kronheimer}, and there exists a residual $F^c$ action on both sides with respect to which $\tilde{\Psi}$ is equivariant. 

\end{lem}

\begin{proof}
By Lemma \ref{transaction2}, we know that $\G^{F,\Gamma}_{0,\C}$ acts freely and transitively on the space of connections. Hence, we can take $\tilde{\Psi}$ to be the following composition of maps: let $\mathcal{C}$ denote $\{(A_0+B,\Theta)\in \A^F\times C^{\infty}(E(\Gamma))|\bar{\partial}_{A_0+B}\Theta=0\}$, and consider
$$\tilde{\Psi}: M \to \mathcal{C} \to \mathcal{C}/\G^{F,\Gamma}_{0,\C},$$$$(\alpha,\beta)=\lambda\mapsto (A_0,\Theta: x\mapsto[p,\lambda(p)])\mapsto [(A_0,\Theta: x\mapsto[p,\lambda(p)])],$$ where $[(A_0,\Theta: x\mapsto[p,\lambda(p)])]$ denotes the gauge orbit containing the chosen representative. Previous arguments suggest that $\tilde{\Psi}$ is an isomorphism. It follows naturally that $\tilde{\Psi}$ is equivariant with respect to the residual $F^c$ action on both $M$ and $ \mathcal{C}/\G^{F,\Gamma}_{0,\C}$. 

\end{proof}

\begin{rmk}\label{psimaptau}
We let $\tilde{\Psi}_\tau$ denote the map $$\tilde{\Psi}_\tau:M\to \{(A_0+B,\Theta)\in \A^F_\tau \times C^{\infty}(E(\Gamma))|\bar{\partial}_{A_0+B}\Theta=0\}/\G^{F,\Gamma}_{\tau,0,\C}. $$
We have that $\tilde{\Psi}_\tau$ is again an isomorphism following the same arguments as in the previous lemma.
\end{rmk}

Before proceeding, we set up some linear algebra that will be of use later. Recall that $E(V)$ is the vector bundle associated to $S^3$ on the $\Gamma$-representation
$V = End(R)$. We have the following two maps induced by left and right multiplication on $V$: $$c_l: V \to End(V), c_l(\phi)(\psi)=\phi\circ\psi$$ and $$c_r: V \to End(V), c_l(\phi)(\psi)=\psi\circ\phi.$$ Since both $c_l$ and $c_r$ commute with the $S^1$-action, they give rise to bundle maps: $$c_l,c_r: E(V)\to E(End(V)).$$ Hence, given $\phi,\psi \in E(V)_x$, we have the following composition: $$E(V)_x\otimes E(V)^*_x\to E(End(V))_x\otimes E(End(V))^*_x\to {\underline{End(V)}}_x,$$ $$\phi\otimes\psi^*\mapsto c_l(\phi)\otimes c_l(\psi^*)\mapsto[c_l(\phi),c_l(\psi^*)].$$ On the other hand, we also have $$E(V)_x\otimes E(V)^*_x\to \underline{End(R)}_x\to \underline{End(End(R))}_x=\underline{End(V)}_x,$$ $$\phi\otimes\psi^*\mapsto[\phi,\psi^*]\mapsto c_l([\phi,\psi^*]),$$ where we also have $$[c_l(\phi),c_l(\psi^*)]=c_l([\phi,\psi^*]).$$

Similarly, there are maps such as $$E(End(R))\otimes \underline{End(R)}\to E(End(R)),$$ $$\underline{End(R)}\otimes E(End(R))\to E(End(R)),$$ $$\underline{End(R)}\otimes E(End(R))\otimes\underline{End(R)}\to E(End(R)),$$ modeled locally on maps such as $$End(R)\otimes End(R)\to End(R), \phi\otimes\psi\mapsto \phi\circ\psi.$$

\begin{lem}[Uniqueness theorem 1]\label{uniqueness1}
Let $(B_1,\Theta_1)$ and $(B_2,\Theta_2)$ be two solutions to \ref{holo} and \ref{eq1}  in $\A^F\times C^\infty(E(\Gamma))$ that lie on the same complex orbit, that is, there exists a complex automorphism of $E(\Gamma)$ taking $(B_1,\Theta_1)$ to $(B_2,\Theta_2)$. Then $(B_1,\Theta_1)$ and $(B_2,\Theta_2)$ are unitarily equivalent.
\end{lem}

\begin{proof}
This proof is modeled on Hitchin's proof of Theorem (2.7) in \cite{hitchin}. Let $\kappa: E(\Gamma) \to E(\Gamma)$ be the complex automorphism satisfying $ \Theta_1\kappa=\kappa\Theta_2$ and $\bar{\partial}_{B_1}\kappa=\kappa\bar{\partial}_{B_2}$. We also have $$\bar{\partial}_{A_0+B_1}\Theta_1=\bar{\partial}_{A_0+B_2}\Theta_2=0$$ and $$F_{B_1}-\frac{i}{2}[\Theta_1,\Theta_1^*]\omega_{vol}=F_{B_2}-\frac{i}{2}[\Theta_2,\Theta_2^*]\omega_{vol}=\sigma.$$ Now we define two bundles: let $$W=End(E(\Gamma))\cong E(\Gamma)\otimes E(\Gamma)^*,$$ and let $$W^\circ=E(End(V))^\Gamma.$$ We remark that both $W$ and $W^\circ$ have the same fibers isomorphic to $End(V)$, but $W$ is a trivial bundle whereas $W^\circ$ is again an associated bundle of $S^3$. We can think of $\kappa$ as a section of $W$.
We also have that  $\Theta_1$ and $\Theta_2$ together define a section $$\boldsymbol\Theta=c_l(\Theta_1)-c_r(\Theta_2)$$ of $W^\circ$, and $B_1$ and $B_2$ together define a connection $$\boldsymbol B=B_1\otimes id-id\otimes B_2^{*}$$ on both $W$ and $W^\circ$, as $End(W)$ and $End(W^\circ)$ are both isomorphic to $End(End(V))^\Gamma$. As we have $$\kappa\Theta_1=\Theta_2\kappa,$$ we must have that $$\boldsymbol\Theta\kappa=(c_l(\Theta_1)-c_r(\Theta_2))\kappa=0.$$

We observe that  the pair $(\boldsymbol B, \boldsymbol\Theta)$ satisfies the equations $$\bar{\partial}_{\boldsymbol B}\boldsymbol\Theta=0$$ and $$F_{\boldsymbol B}-\frac{i}{2}[\boldsymbol\Theta,\boldsymbol\Theta^*]\omega_{vol}=\ad(\sigma)$$  on $W^\circ$, where $[\boldsymbol\Theta,\boldsymbol\Theta^*]=c_l([\Theta_1,\Theta_1^*])-c_r([\Theta_2,\Theta_2^*])$.

To proceed, we now think of $\kappa$ as a holomorphic section of $W$ with respect to $\boldsymbol B$, that is, $\bar{\partial}_{\boldsymbol{B}}\kappa=0$, as $\bar{\partial}_{B_1}\kappa=\kappa\bar{\partial}_{B_2}$. Before continuing further, we first prove a useful identity. Consider $$\bar{\partial}\langle \partial_{\boldsymbol B}\kappa,\kappa\rangle=\langle \bar{\partial}_{\boldsymbol B}\partial_{\boldsymbol B}\kappa,\kappa\rangle-\langle \partial_{\boldsymbol B}\kappa,\partial_{\boldsymbol B}\kappa\rangle.$$ Since $F_{\boldsymbol B}=\bar{\partial}_{\boldsymbol B}\partial_{\boldsymbol B}+\partial_{\boldsymbol B}\bar{\partial}_{\boldsymbol B}$ and $\bar{\partial}_{\boldsymbol B}\kappa=0$, we have $$\bar{\partial}\langle \partial_{\boldsymbol B}\kappa,\kappa\rangle=\langle F_{\boldsymbol B}\kappa,\kappa\rangle-\langle \partial_{\boldsymbol B}\kappa,\partial_{\boldsymbol B}\kappa\rangle.$$ Now we integrate on both sides and get
$$\int_{S^2/\Gamma}\bar{\partial}\langle \partial_{\boldsymbol B}\kappa,\kappa\rangle+\int_{S^2/\Gamma}\langle \partial_{\boldsymbol B}\kappa,\partial_{\boldsymbol B}\kappa\rangle=\int_{S^2/\Gamma}\langle F_{\boldsymbol B}\kappa,\kappa\rangle,$$ and by Stokes' theorem, we get $$0 \leq \int_{S^2/\Gamma}\langle \partial_{\boldsymbol B}\kappa,\partial_{\boldsymbol B}\kappa\rangle =\int_{S^2/\Gamma}\langle F_{\boldsymbol B}\kappa,\kappa\rangle.$$

Hence, we have $$\int_{S^2/\Gamma}\langle \partial_{\boldsymbol B}\kappa,  \partial_{\boldsymbol B}\kappa\rangle=\int_{S^2/\Gamma}\langle F_{\boldsymbol B}\kappa,\kappa\rangle=$$$$\int_{S^2/\Gamma}\frac{i}{2}\langle[\boldsymbol\Theta,\boldsymbol\Theta^*]\kappa,\kappa\rangle\omega_{vol}-\int_{S^2/\Gamma}\langle\ad(\sigma)\kappa,\kappa\rangle.$$ Since $\sigma$ takes values in the center $Z$, we have that $\kappa$ commutes with $\sigma$, i.e., $\ad(\sigma)\kappa=0$, and hence the following equation $$-\int_{S^2/\Gamma}\langle\ad(\sigma)\kappa,\kappa\rangle=0$$ holds as $\sigma\otimes1(\kappa)=1\otimes\sigma^T(\kappa)$, which can be shown using  essentially the same arguments as in showing $\boldsymbol\Theta\kappa=0$.

As we have shown that $\boldsymbol\Theta \kappa=0$, we also obtain $$\langle[\boldsymbol\Theta,\boldsymbol\Theta^*]\kappa,\kappa\rangle=\langle \boldsymbol\Theta\boldsymbol\Theta^*\kappa,\kappa\rangle=\langle \boldsymbol\Theta^*\kappa,\boldsymbol\Theta^*\kappa\rangle\geq 0$$ and hence must be purely real. Consequently, $\frac{i}{2}\langle[\boldsymbol\Theta,\boldsymbol\Theta^*]\kappa,\kappa\rangle$ must be purely imaginary, so it must be $0$. This gives us that $\partial_{\boldsymbol{B}}\kappa=0$.

Putting everything together, we have $\partial_{\boldsymbol B}\kappa=\bar{\partial}_{\boldsymbol B}\kappa=0$, $\boldsymbol\Theta \kappa=\boldsymbol\Theta^*\kappa=0$. Let $\rho=\kappa(\kappa^*\kappa)^{-\frac{1}{2}}$ then we must have $d_{\boldsymbol B}\rho=0$. Since $\boldsymbol\Theta \kappa=\boldsymbol\Theta^*\kappa=0$, we have $\kappa^*\Theta_2=\Theta_1\kappa^*$ and $\kappa\Theta_2=\Theta_1\kappa$, which implies $\rho\Theta_2=\Theta_1\rho$. Hence, we obtain the desire statement that $(B_1,\Theta_1)$ and $(B_2,\Theta_2)$ lie on the same unitary gauge orbit.

\end{proof}

\begin{cor}[Uniqueness theorem 2]\label{uniqueness2}
Let $(B_1,\Theta_1)$ and $(B_2,\Theta_2)$ be two solutions to \ref{holo} -- \ref{mu3} in $\A^F_\tau\times C^\infty(E(\Gamma))$ that lie on the same $\G^{F,\Gamma}_{\tau,\C}$-orbit, that is, there exists a complex automorphism of $E(\Gamma)$ in $\G^{F,\Gamma}_{\tau,\C}$ that takes $(B_1,\Theta_1)$ to $(B_2,\Theta_2)$. Then $(B_1,\Theta_1)$ and $(B_2,\Theta_2)$ lie on the same $\G^{F,\Gamma}_\tau$-orbit.
\end{cor}

\begin{proof}
Let $\kappa$ be such a complex automorphism. By the same arguments as in the previous lemma, we can modify $\kappa$ and obtain a unitary gauge element $\rho=\kappa(\kappa^*\kappa)^{\frac{1}{2}}$ that also sends $(B_1,\Theta_1)$ to $(B_2,\Theta_2)$. We must also have that $\rho$ is $\tau$-invariant as $\kappa$ is $\tau$-invariant. Hence, $\rho$ lies in $\G^{F,\Gamma}_\tau$.
\end{proof}

Before we proceed to the next section, we prove the following proposition which analyzes the stabilizer group of a holomorphic section $\Theta$.

\begin{prop}\label{trivstab}
If $\Theta$ has trivial stabilizer in $Stab(B)$ with $\bar{\partial}_{A_0+B}\Theta=0$, then $\Theta$ has trivial stabilizer in $\G^{F,\Gamma}$.
\end{prop}

\begin{proof}
Let $\kappa: E(\Gamma)\to E(\Gamma)$ be a complex automorphism on $E(\Gamma)$ taking $A_0$ to $A_0+B$. In other words, we have $B=\kappa^{-1}\bar{\partial}\kappa+\kappa^*\partial \kappa^{*-1}$. Consider $\kappa^{-1}\Theta \kappa$, it is a holomorphic section of $E(\Gamma)$ with respect to $A_0$. Hence, we can rewrite $\kappa^{-1}\Theta \kappa$ as a pair of matrices $(\alpha,\beta)$. The identification is as follows: for $x \in S^2$, $\kappa^{-1}\Theta \kappa: x\mapsto [p,\lambda(p)]$, where $\lambda: S^3\to End(R)$ is given by $\lambda(z_1,z_2)=z_1\alpha+z_2\beta$. 

Since $(\alpha,\beta)$ is $\Gamma$-invariant, we have that for $\gamma=\begin{pmatrix}
u & v \\
-v^* & u^*  
\end{pmatrix}$, the pair $(\alpha,\beta)$ must satisfy \begin{equation}\label{id1}R(\gamma^{-1})\alpha R(\gamma)=u\alpha+v\beta\end{equation} and \begin{equation}\label{id2}R(\gamma^{-1})\beta R(\gamma)=-v^*\alpha+u^*\beta\end{equation}  as in \cite{kronheimer}. Notice that if $v\neq 0$, then $\beta$ is uniquely given by $\beta=v^{-1}R(\gamma^{-1})\alpha R(\gamma)-v^{-1}u\alpha$. On the other hand, if $v=0$ for all $\gamma \in \Gamma$, then it implies that $\Gamma$ is a cyclic subgroup. Hence, we break the proof into two cases.
\smallskip

\emph{Case 1:} $\Gamma$ is not cyclic.

In this case, we have that $v\neq0$ and $\beta=v^{-1}R(\gamma^{-1})\alpha R(\gamma)-v^{-1}u\alpha$. First, we want to show that $(\alpha,\beta)$ has trivial stabilizer in $F/T$ if and only if $\alpha$ has trivial stabilizer in $F/T$. We can assume that $\alpha$ and $\beta$ are both nonzero as by \ref{id1} and \ref{id2}, it's easy to see that if either $\alpha$ or $\beta$ is $0$, then both have to be $0$.  

We observe that $(\alpha,\beta)$ has trivial stabilizer in $F/T$ if and only if $\alpha$ has trivial stabilizer in $F/T$: if $\alpha$ has trivial stabilizer in $F/T$, then clearly $(\alpha,\beta)$ has trivial stabilizer in $F/T$; on the other hand, if some element $f$ stabilizes $\alpha$, then it stabilizes $\beta$ as well by the equality $\beta=v^{-1}R(\gamma^{-1})\alpha R(\gamma)-v^{-1}u\alpha$, so $f$ stabilizes $(\alpha,\beta)$. With the preceding arguments, we can rephrase the assumption that $(\alpha,\beta)$ has trivial stabilizer in $F/T$ as simply that $\alpha$ has trivial stabilizer in $F/T$. 

Now, at a point $p$ thought of as a pair $(z_1,z_2)$, we can use some $\gamma\in\Gamma$ to get the following equality $$f(z_1\alpha+z_2\beta)f^{-1}=f(z_1\alpha-z_2v^{-1}u\alpha+z_2v^{-1}R(\gamma^{-1})\alpha R(\gamma))f^{-1}$$$$=z_1f\alpha f^{-1}-z_2v^{-1}uf\alpha f^{-1}+z_2v^{-1}R(\gamma^{-1})(f\alpha f^{-1})R(\gamma).$$ Assume that we are given $f\alpha f^{-1}\neq \alpha$, for all $f \in F/T$, we want to show that for any pair of points $(z_1,z_2)$ and for all $f\in F/T$, we always have the following: $$z_1\alpha-z_2v^{-1}u\alpha+z_2v^{-1}R(\gamma^{-1})\alpha R(\gamma)\neq z_1f\alpha f^{-1}-z_2v^{-1}uf\alpha f^{-1}+z_2v^{-1}R(\gamma^{-1})(f\alpha f^{-1})R(\gamma).$$ To achieve this end, let $L_\gamma$ be the linear map defined as follows: for a pair $(c,d)\in M$, consider $$L_\gamma: c \mapsto z_1c-z_2v^{-1}uc+z_2v^{-1}R(\gamma^{-1})c R(\gamma).$$ Then we need to show $L_\gamma(\alpha)\neq L_\gamma(f\alpha f^{-1})$. As we know that $\alpha\neq f\alpha f^{-1}$, it suffices to show that $$\bigcap_{\gamma\in \Gamma} \ker(L_\gamma)=0.$$  We can assume that $z_2\neq 0$ as the inequality is clearly satisfied when $z_2=0$. Hence, $c$ lies in the kernel of $L_\gamma$ if $$\frac{z_2v^{-1}u-z_1}{z_2v^{-1}}c=R(\gamma^{-1})c R(\gamma).$$ This implies that $c$ must be a scalar multiple of $d$, that is, $d=q c$; in particular, by applying \ref{id1} and \ref{id2} to the pair $(c,d)$, we must have $(qu+q^2v+v^*-u^*q)c=0$. Notice that this equality must be satisfied for any choice of $\gamma\in\Gamma$ with $v\neq0$, and since $q$ and $c$ are fixed, we see that this equality can only hold when $c=0$. As a result, $z_1\alpha+z_2\beta$ has trivial stabilizer for all $(z_1,z_2)$, which gives us that $(\alpha,\beta)$ has trivial stabilizer in $\G^{F,\Gamma}$.

\smallskip
\emph{Case 2:} $\Gamma$ is cyclic.

When $\Gamma$ is a cyclic subgroup, we can write down $\alpha$ and $\beta$ explicitly and describe the action of $\Gamma$ and $F/T$ explicitly as well. We use the decomposition of $M$ in terms of simply-laced Dynkin diagram given in \cite{kronheimer} and reviewed in Section 2.1: $$M=\bigoplus_{i,j}a_{ij}Hom(\C^{n_i},\C^{n_j}).$$ We also have that $$F=\times_iU(n_i).$$ For the case where $\Gamma$ is cyclic, $n_i=1$ for all $i$, and $$M=(\bigoplus_{i}Hom(\C^{n_i},\C^{n_{i+1}}))\oplus(\bigoplus_{j}Hom(\C^{n_{j+1}},\C^{n_{j}})).$$

We can regard $\alpha \in\bigoplus_{i}Hom(\C^{n_i},\C^{n_{i+1}})$ and $\beta\in\bigoplus_{j}Hom(\C^{n_{j+1}},\C^{n_{j}})$. Hence, we can write $\alpha=(a_1,...,a_n)$ and $\beta=(b_1,...,b_n)$, and $F$ acts on $\C^{n_i}$ and $\C^{n_j}$ by scalar multiplification.

For $(\alpha,\beta)$ to have trivial stabilizer in $F/T$, we must have that for all $i\in \{1,...,n\}$, at least one of $a_i$ and $b_i$ is not $0$. For $z_1\alpha+z_2\beta$ to have trivial stabilizer in $F/T$ at $(z_1,z_2)$, we must have that for all $i\in \{1,...,n\}$, at least one of $z_1a_i$ and $z_2b_i$ is not $0$. But this can only happen when either $z_1$ or $z_2$ is $0$. This means that the stabilizer of $(\alpha,\beta)$ in $\G^{F,\Gamma}$ must be the identity away from $(0,z_2)$ and $(z_1,0)$, and hence it must be the identity by continuity.

Hence, we have shown that if $\lambda(p)$ has trivial stabilizer at a single $p$, then for any other $p'$, $\lambda(p')$ also has trivial stabilizer. This is equivalent to saying that if $\kappa^{-1}\Theta \kappa$ has trivial stabilizer in $F/T$, then it has trivial stabilizer in $\G^{F,\Gamma}$. By pushing forward using $\kappa$, we get the desired statement of the lemma. 

\end{proof}

\begin{cor}\label{trivstabtau}
If $\Theta$ has trivial stabilizer in $Stab(B)$ with $\bar{\partial}_{A_0+B}\Theta=0$, then $\Theta$ has trivial stabilizer in $\G^{F,\Gamma}_\tau$.
\end{cor}

\section{Smoothness and dimension calculations}

In this section, we show that the moduli space is a smooth finite-dimensional manifold and calculate its dimension which will be useful for proving Theorem \ref{main}. We refer the readers to \cite{sobolev-space} for the related background material on elliptic operators and Sobolev spaces. To proceed, we first introduce the following lemma.

\begin{lem}\label{analysis}
If $(B,\Theta)$ and $(B',\Theta')$ are two solutions to \ref{holo} and \ref{eq1}  in $\A^F\times C^\infty(E(\Gamma))$ with $B$ and $B'$ not $\G^{F,\Gamma}$-equivalent, then $(B',\Theta')$ is separated from the subset of solutions such that the connection part is $\G^{F,\Gamma}$-equivalent to $B$. 
\end{lem}

\begin{proof}
Suppose we have two solutions $(B,\Theta)$ and $(B',\Theta')$ such that $B$ is not $\G^{F,\Gamma}$-equivalent to $B'$. We proceed by contradiction. Suppose that there exists a sequence of solutions $\{(B_n,\Theta_n)\}_n$ such that $(B_1,\Theta_1)=(B,\Theta)$ and $\{(B_n,\Theta_n)\}_n$ converges weakly in $L^2_1$ to $(B',\Theta')$ with $B_n$ lying on the same  $\G^{F,\Gamma}$-orbit as $B$, for all $n$. Then we get a sequence of gauge elements lying in $\G^{F,\Gamma}$, denoted by $\{\rho_n\}$, such that $\rho_n\cdot B=B_n$, for all $n$. (Note that we don't assume $ \rho_n\cdot\Theta=\Theta_n$.) We want to show that $\{\rho_n\}$ converges weakly to some $\rho$. To this end, we follow Hitchin's proof of Theorem (2.7) in \cite{hitchin}. Consider the following: $$\bar{\partial}_{B_1B_n}:\Omega^0(S^2/\Gamma;E(\Gamma)^*\otimes E(\Gamma))\to\Omega^{0,1}(S^2/\Gamma;E(\Gamma)^*\otimes E(\Gamma)),$$ where $B_n$ acts on the $E(\Gamma)^*$ factor, and $B_1$ acts on the $E(\Gamma)$ factor. Hence, $\bar{\partial}_{B_1 B'}=\bar{\partial}_{B_1B_n}+t_n$ where $t_n\to 0$ weakly in $L^2_1$. As before, $\rho_n$ is the sequence of unitary gauge elements taking $B$ to $B_n$, and $\lVert\rho_n\rVert_{L^2}=1$.

We also have $$\rho_n\cdot B_1=\rho_n^*\circ\partial_{B_1}\circ\rho_n^{*-1}+\rho_n^{-1}\circ\bar{\partial}_{B_1}\circ\rho_n=\partial_{B_n}+\bar{\partial}_{B_n}.$$ Hence, $\rho_n^*\circ\partial_{B_1}\circ\rho_n^{*-1}=\partial_{B_n}$ and $\rho_n^{-1}\circ\bar{\partial}_{B_1}\circ\rho_n=\bar{\partial}_{B_n}$, so we have $\bar{\partial}_{B_1}\circ\rho_n-\rho_n\circ\bar{\partial}_{B_n}=0$, but this is equivalent to $\bar{\partial}_{B_1B_n}\rho_n=0$. 

Now, the elliptic estimate for $\bar{\partial}_{B_1B_n}$ gives us $$\lVert\rho_n\rVert_{L^2_1}\leq C(\lVert[t_n,\rho_n]\rVert_{L^2}+\lVert\rho_n\rVert_{L^2})=C(\lVert[t_n,\rho_n]\rVert_{L^2}+1)\leq K_1\lVert t_n\rVert_{L^4}\lVert\rho_n\rVert_{L^4}+K_2.$$ Since $L^2_1\subset L^4$ compactly, we have that $\lVert\rho_n\rVert_{L^2_1}$ is bounded and hence has a weakly convergent subsequence. Since $L^2_1 \subset L^2$ is compact and $\lVert\rho_n\rVert_{L^2}=1$, the weak limit $\rho$ is non-zero. 

Hence, we have $\rho\cdot B=B'$. Since by construction, $B$ and $B'$ lie on the same complex orbit, $\rho$ must be a complex automorphism. Now since weak convergence implies pointwise convergence, that is, $\rho_n(x)\to\rho(x)$, for all $x\in S^2/\Gamma$, and $F/T$ is compact, we must have $\rho(x)\in F/T$, for all $x$. Hence, $\rho$ lies in $\G^{F,\Gamma}$, but this is a contradiction.

\end{proof}

\begin{cor}\label{analysistau}
If $(B,\Theta)$ and $(B',\Theta')$ are two solutions to \ref{holo} -- \ref{mu3}  in $\A^F_\tau\times C^\infty(E(\Gamma))$ with $B$ and $B'$ not $\G^{F,\Gamma}_\tau$-equivalent, then $(B',\Theta')$ is separated from the subset of solutions such that the connection part is $\G^{F,\Gamma}_\tau$-equivalent to $B$. 
\end{cor}

\begin{proof}
We assume otherwise and again follow the same arguments as in the previous lemma with the further assumption that all the gauge transformations are $\tau$-invariant, that is, they lie in $\G^{F,\Gamma}_\tau$. Hence, we obtain a limit $\rho$ that lies in $\G^{F,\Gamma}_\tau$ and hence obtains a contradiction. 

\end{proof}

\begin{cor}\label{conncomp}
\begin{enumerate}
\item
Solutions to \ref{holo} and \ref{eq1} in  $\A^F\times C^\infty(E(\Gamma))$ with the connection part $B$ not $\G^{F,\Gamma}$-equivalent lie in different connected components of the moduli space $\M(\Gamma,\tilde{\zeta}_1)$. 
\item
Solutions to \ref{holo} -- \ref{mu3} in  $\A^F_\tau\times C^\infty(E(\Gamma))$ with the connection part $B$ not $\G^{F,\Gamma}_\tau$-equivalent lie in different connected components of the moduli space $\X_{\tilde{\zeta}}$. 
\end{enumerate}
\end{cor}

\begin{prop}\label{dimcalc}
\begin{enumerate}
\item
Suppose $(B,\Theta)$ is a solution to \ref{holo} and \ref{eq1} in $\A^F\times C^\infty(E(\Gamma))$ with trivial stabilizer in $\G^{F,\Gamma}$, the moduli space $\M(\Gamma, \tilde{\zeta}_1)$ at the orbit of $(B,\Theta)$ is smooth of dimension $2|\Gamma|+2$.
\item
If $(B,\Theta)$ is a solution to \ref{holo} -- \ref{mu3} in $\A^F_\tau\times C^\infty(E(\Gamma))$ with trivial stabilizer in $\G_\tau^{F,\Gamma}$, the moduli space $\X_{\tilde{\zeta}}$ at the orbit of $(B,\Theta)$ is smooth of dimension $4$. 
\end{enumerate}
\end{prop}

\begin{proof}
\begin{enumerate}
\item
Consider the set of sections $\mathcal{S}=\{\Theta \in C^\infty{E(\Gamma)}| \bar{\partial}_{A_0+B}\Theta=0\}$. The stabilizer group $Stab(B)$ of $B$ in $\G^{F,\Gamma}$ acts on $\mathcal{S}$. By Lemma \ref{transaction2} and Lemma \ref{psimap} (with small adaptations of the proof), we have that $\mathcal{S}$ is isomorphic to $M=P^\Gamma$ and $Stab(B)$ is isomorphic to $F$. Hence, we can restrict the symplectic structure compatible with $I$ on $C^\infty{E(\Gamma)}$ to $\mathcal{S}$ and obtain a Hamiltonian action of $Stab(B)$ on $\mathcal{S}$ with respect to the restrictions of $I$ on $\mathcal{S}$. We also know that $Stab(B)$ acts freely at $\Theta\in\mathcal{S}$ as $\G^{F,\Gamma}$ acts freely at $(B,\Theta)$. On the other hand, by Lemma \ref{analysis} and Corollary \ref{conncomp}, every point in the connected component of $\M(\Gamma, \tilde{\zeta}_1)$ containing the orbit of $(B,\Theta)$ has a unique representative lying in $\mathcal{S}$. Hence, the smoothness and the dimension of $\M(\Gamma, \tilde{\zeta}_1)$  at $[(B,\Theta)]$ follow from Proposition \ref{hkred} (cf. Proposition 2.1 in \cite{kronheimer}).

\item
First, we observe that the action of $J$ commutes with the action of $\rho$ when $\rho$ lies in $\G^{F,\Gamma}_\tau$. Hence, we can restrict the hyperk{\"a}hler structure on $C^\infty{E(\Gamma)}$ to $\mathcal{S}$ and obtain a Hamiltonian action of $Stab(B)$ on $\mathcal{S}$ with respect to the restrictions of $I$, $J$, and $K$ on $\mathcal{S}$. We also know that $Stab(B)$ acts freely at $\Theta\in\mathcal{S}$ as $\G^{F,\Gamma}_\tau$ acts freely at $(B,\Theta)$. On the other hand, by Corollary \ref{analysistau} and Corollary \ref{conncomp}, every point in the connected component of $\X_{\tilde{\zeta}}$ containing the orbit of $(B,\Theta)$ has a unique representative lying in $\mathcal{S}$. Hence, the smoothness and the dimension of  $\X_{\tilde{\zeta}}$  at $[(B,\Theta)]$ again follow from Proposition \ref{hkred} (cf. Proposition 2.1 in \cite{kronheimer}). 
\end{enumerate}
\end{proof}

\section{Proof of Theorem \ref{main}}
\subsection{A criterion for obtaining free $\G^{F,\Gamma}_\tau$-action}
Now we want to give a criterion for when the $\G^{F,\Gamma}_\tau$-action is free on $\tilde{\mu}^{-1}(\tilde{\zeta})$.

We adapt the notations introduced in \cite{kronheimer} and in Section 2.1 to our setting. Consider projection maps $$\pi_i: R\to \C^{n_i}\otimes R_i.$$ Now, let $\hat{Z}$ denote the center of $\f$. Then $\Omega^0(S^2/\Gamma;\hat{Z})$ is spanned by elements $\sqrt{-1}\pi_i$, that is, smooth sections such that at each point the endomorphism is a scalar multiple of the projection map. Let $h$ denote the real Cartan algebra associated to the Dynkin diagram, we then get a linear map $l$ from $\Omega^0(S^2/\Gamma;\hat{Z})$ to $\Omega^0(S^2/\Gamma;h^*)$ given by $$l: \sqrt{-1}\pi_i \mapsto n_i\xi_i,$$ and hence $l$ induces a map $\tilde{l}$ from $\Omega^0(S^2/\Gamma;Z)$ to $\Omega^0(S^2/\Gamma;h)$ which is an isomporhism. 

Let $\xi$ be a root, not necessarily simple. We define $\tilde{D}_\xi$ to be $\ker(\xi\circ\tilde{l})$, where we regard $\xi$ as a constant element in $\Omega^0(S^2/\Gamma;h^*)$.

\begin{lem}\label{freeaction}
Let $(B,\Theta)$ be a solution to \ref{holo} -- \ref{mu3} in  $\A^F_\tau\times C^\infty(E(\Gamma))$. If $\G^{F,\Gamma}_\tau$ does not act freely on $(B,\Theta)$, then $\tilde{\zeta}$ lies in $\R^3\otimes {\tilde{D}}_{\xi}$.
\end{lem}

\begin{proof}

This proof is a reformulation of Kronheimer's original proof of Proposition 2.8 in \cite{kronheimer} in our setting. Suppose $(B,\Theta)\in \mu^{-1}(\tilde{\zeta})$ is fixed by some $\rho\in \G^{F,\Gamma}_\tau$. In particular, $\rho$ lies in $Stab(B)$ and fixes $\Theta$. Then we can rewrite $\rho$ as $$\rho=\kappa\rho_0\kappa^{-1},$$ where $\rho_0$ is a constant in the complexification of $F/T$ and $$\kappa:E(\Gamma)\to E(\Gamma)$$ is a complex automorphism with $$\kappa^{-1}\bar{\partial}\kappa+\kappa^*\partial \kappa^{*-1}=B.$$ We can find a lift $\tilde{\rho_0}$ of $\rho_0$ in the complexification of $F$ and decompose $R$ into the eigenspaces of $\tilde{\rho_0}$ and obtain at least two $\Gamma$-invariant parts $$R=R'\oplus R''.$$ We have that $E(End(R'))$ is naturally a holomorphic subbundle of $E(\Gamma)$ with respect to $A_0$. This gives rise to a holomorphic subbundle $\tilde{E}$ of $E(\Gamma)$ with respect to $B$ where the fiber of $\tilde{E}$ over each point $x$ is isomorphic to $End(R')$. Explicitly, $\tilde{E}$ is the image of $E(End(R'))$ under $\kappa$. 

Without loss of generality we assume that $\Theta$ is a holomorphic section of $\tilde{E}$ with a free action by $Map(S^2/\Gamma, F'/T')$, where $Map(S^2/\Gamma, F'/T')$ is the natural gauge group acting on $\tilde{E}$. In other words, $\tilde{E}$ is the smallest holomorphic subbundle of $E(\Gamma)$ such that $\Theta$ is a holomorphic section of $\tilde{E}$ and there is no proper subbundle of $\tilde{E}$ of which $\Theta$ is a section. We observe that $\tilde{E}$ is $\Gamma$-invariant. In particular, $\tilde{E}$ is isomorphic to $E(End(R'))^\Gamma$.

By Proposition \ref{dimcalc}, we know that the condition that $Map(S^2/\Gamma, F'/T')$ acts freely on $\Theta$ means that the moduli space of the reduction by $Map(S^2/\Gamma, F'/T')$ on pairs on $\tilde{E}$ is a smooth manifold at at least one point, with dimension $$\dim_\R (Hom(\C^2, End(R'))^\Gamma)-4\dim_\R (F'/T')\geq 0.$$ This translates to $$\dim_\C (Hom(\C^2, End(R'))^\Gamma)-2\dim_\C (End(R')^\Gamma)+2\geq 0,$$ and hence we have $$2\dim_\C (End(R')^\Gamma)-\dim_\C (Hom(\C^2, End(R'))^\Gamma)\leq 2.$$

Now further decompose $R'$ into irreducibles $R'=\oplus n_i'R_i$, then the above inequality is the same as the following: $$2\sum_i(n_i')^2-\sum_{i,j}a_{i,j}n_i'n_j'\leq2.$$ Equivalently, $$\sum_{i,j}c_{i,j}n_i'n_j' \leq 2,$$ where $\bar{C}=(c_{i,j})$ is the extended Cartan matrix. Now let $\xi$ be defined by $$\xi=\sum_{0}^rn_i'\xi_i.$$ The inequalities suggest that $$\rVert\xi\rVert^2\leq 2,$$ which implies that $\xi$ is a root. 

Let $\pi_B: E(\Gamma) \to \tilde{E}$ be the projection from $E(\Gamma)$ to $\tilde{E}$. We then have that $\pi_B$ induces an element $\tilde{\pi}\in \Omega^0(S^2/\Gamma;\f)$ such that  $\tilde{\pi}(x)\in End(R)$ is given by $$\tilde{\pi}(x): R_x\to R'_x,$$ where $R_x$ is isomorphic to $R$, and $R'_x$ is a subrepresentation of $R_x$ which is also isomorphic to $R'$, for all $x$. Notice that $\tilde{\pi}$ is identified with $\kappa\cdot \xi=\kappa\xi\kappa^{-1}=\xi$ under $l$, as $\xi$ is in the center.

We have that $\tilde{\pi}$ acts trivially on $\Theta$, that is, $[\tilde{\pi},\Theta]=0$, as it is the identity on $\tilde{E}$. Now consider $\tilde{\zeta}(\tilde{\pi})$. We compute $\tilde{\zeta}_1(\tilde{\pi})$ here: $$\tilde{\zeta}_1(\tilde{\pi})=\int_{S^2/\Gamma}\tr(\tilde{\pi} F_B)-\frac{i}{2}\int_{S^2/\Gamma}\tr(\tilde{\pi}[\Theta,\Theta^*])\omega_{vol}.$$
We know that $$\int_{S^2/\Gamma}\tr(F_{A_0+B})=\int_{S^2/\Gamma}\tr(F_{A_0})+\tr(F_{B})=\frac{i}{2\pi}\cdot c_1(E(\Gamma)).$$ By construction, the integral of $c_1(E(\Gamma))$ concentrates on $A_0$, that is, $$\int_{S^2/\Gamma}\tr(F_{A_0})=\frac{i}{2\pi}\cdot c_1(E(\Gamma)).$$ Hence, we have that $\int_{S^2/\Gamma}\tr(F_{B})=0$. Since $\tilde{E}$ is a holomorphic subbundle of $E(\Gamma)$ and $\tilde{\pi} F_B$ is the projection of $F_B$ onto $\tilde{E}$, we must have that on the subbundle $\tilde{E}$, $$\int_{S^2/\Gamma}\tr(\tilde{\pi}F_{A_0+B})=\int_{S^2/\Gamma}\tr(\tilde{\pi}F_{A_0})+\tr(\tilde{\pi}F_{B})=\frac{i}{2\pi}\cdot c_1(\tilde{E})$$$$=\int_{S^2/\Gamma}\tr(\tilde{\pi}F_{A_0}).$$ Hence, $\int_{S^2/\Gamma}\tr(\tilde{\pi} F_{B})=0$.

We have shown that the first integrand is 0. On the other hand, since $[\tilde{\pi},\Theta]=0$, we have $$\tr(\tilde{\pi}[\Theta,\Theta^*])=\tr(\tilde{\pi}\Theta\Theta^*-\tilde{\pi}\Theta^*\Theta)$$$$=\tr(\tilde{\pi}\Theta\Theta^*-\Theta \tilde{\pi}\Theta^*)=0.$$ Hence, $\tilde{\zeta}_1(\tilde{\pi})=0$, that is to say, $\tilde{\zeta}_1\in {\tilde{D}}_{\xi}$. Similarly, $\tilde{\zeta}_2(\tilde{\pi})=\tilde{\zeta}_3(\tilde{\pi})=0$. As a result, we have $\tilde{\zeta}\in \R^3\otimes {\tilde{D}}_{\xi}$.

\end{proof}

\begin{cor}
For $\zeta$ not lying in $D_\xi$ as in \cite{kronheimer} and $\tilde{\zeta}=-\zeta$ thought of as a constant element in $\Omega^2(S^2/\Gamma;Z)$, $\G^{F,\Gamma}_\tau$ acts freely on $\tilde{\mu}^{-1}(\tilde{\zeta})$.
\end{cor}

\begin{proof}
If $\zeta$ doesn't lie in $D_\xi$ as in \cite{kronheimer}, then $\tilde{\zeta}=-\zeta$ thought of as a constant element in $\Omega^2(S^2/\Gamma;Z)$ doesn't lie in $\tilde{D}_{\xi}$. Hence, by the previous lemma, $\G^{F,\Gamma}_\tau$ acts freely on $\tilde{\mu}^{-1}(\tilde{\zeta})$.
\end{proof}

\subsection{Proof of Theorem \ref{main} Part I}

In this subsection, we prove one direction of Theorem \ref{main} where we show the moduli space obtained by the gauge-theoretic construction contains the 4-dimensional hyperk{\"a}hler ALE space given by Kronheimer's construction. To do this, we first explicitly identify certain solutions to the equations given previously with solutions to the equations given in Kronheimer's work and hence show that the moduli space contains the corresponding 4-dimensional hyperk{\"a}hler ALE space. In the following subsection, we will show that by the uniqueness results, smoothness results and dimension calculations, there cannot be any additional solutions other than the ones corresponding to the points of the 4-dimensional hyperk{\"a}hler ALE space. Hence, we identify the moduli space with a 4-dimensional hyperk{\"a}hler ALE space.

\begin{proof}[Proof of Theorem \ref{main} Part I]
\begin{lem}
For $\tilde{\zeta}=\zeta^*=-\zeta$, there is a map $\Phi:X_\zeta\to\X_{\tilde{\zeta}}$ which is an embedding and there is a natural choice of metric on $\X_{\tilde{\zeta}}$ such that $\Phi$ is an isometry onto its image. 
\end{lem}
\begin{proof}
We set $B=0$, then the equations reduce to the following: $$\bar{\partial}_{A_0}\Theta=0,$$$$-\frac{i}{2}[\Theta,\Theta^*]\omega_{vol}=\tilde{\zeta} _1\cdot\omega_{vol}=-\zeta_1\cdot\omega_{vol},$$$$-\frac{1}{4}([J\Theta,\Theta^*]-[\Theta,J\Theta^*])\omega_{vol}=\tilde{\zeta}_2\cdot\omega_{vol}=-\zeta_2\cdot\omega_{vol},$$$$-\frac{i}{4}([J\Theta,\Theta^*]+[\Theta,J\Theta^*])\omega_{vol}=\tilde{\zeta}_3\cdot\omega_{vol}=-\zeta_3\cdot\omega_{vol}.$$
Now since in this case, we can think of $\Theta$ as a pair of matrices $(\alpha,\beta)$, the equations can be further rewritten as the following (here we are implictly dropping the volume $2$-form on both sides):
$$\frac{i}{2}([\alpha,\alpha^*]+[\beta,\beta^*])=\zeta_1$$
$$\frac{1}{2}([\alpha,\beta]+[\alpha^*,\beta^*])=\zeta_2$$
$$\frac{i}{2}([\alpha,\beta]-[\alpha^*,\beta^*])=\zeta_3.$$

These are precisely Kronheimer's moment map equations and hence by the results of Kronheimer, and we get a solution to the equations. By Lemma \ref{uniqueness2}, we know that if a $\G^{F,\Gamma}_{\tau,\C}$-orbit contains a solution coming from $X_\zeta$, it is also the unique solution on that orbit. On the other hand, we also want to argue that two distinct solutions coming from $X_\zeta$ will remain distinct in the new moduli space. Suppose there are two solutions coming from $X_\zeta$ that become identified by $\G^{F,\Gamma}_\tau$, then they must lie on the same $\G^{F,\Gamma}_{\tau,\C}$-orbit as well. Recall that we have $$\{(A_0+B,\Theta)\in \A^F_\tau\times C^{\infty}(E(\Gamma))|\bar{\partial}_{A_0+B}\Theta=0\}/\G^{F,\Gamma}_{\tau,0,\C}\cong M.$$
Hence, two solutions lie on the same $\G^{F,\Gamma}_{\tau,\C}$-orbit if and only if they also lie on the same $F^c$-orbit, which would imply that they are also on the same $F$-orbit. Hence, we define $\Phi$ to be the bottom horizontal map that makes the following diagram commute:
 \begin{center}
\begin{tikzcd}

M \arrow{r}{\Psi} &  \A^F_\tau\times C^{\infty}(E(\Gamma)) \\
\mu^{-1}(\zeta) \arrow{r}{\Psi|_{\mu^{-1}(\zeta)}}\arrow{u}{\iota}\arrow{d}{proj} & \tilde{\mu}^{-1}(\tilde{\zeta}) \arrow{u}{\iota}\arrow{d}{proj}\\
X_\zeta=\mu^{-1}(\zeta)/F\arrow{r}{\Phi}&\X_{\tilde{\zeta}}= \tilde{\mu}^{-1}(\tilde{\zeta})/\G^{F,\Gamma}\end{tikzcd}
\end{center}
That $\Phi$ can be regarded as an isometry onto its image comes from the fact that $\Psi|_{\mu^{-1}(\zeta)}$ is naturally an isometry onto its image, and we can define a metric on $\X_{\tilde{\zeta}}$ as follows: for $[(B_1,\Theta_1)], [(B_2,\Theta_2)] \in \im(\Phi) $, define $$d([(B_1,\Theta_1)],[(B_2,\Theta_2)])=(\inf_{f\in F}\int_{S^2/\Gamma}Re\langle f\Theta'_1f^{-1}-\Theta'_2,f\Theta'_1f^{-1}-\Theta'_2\rangle\omega_{vol})^{\frac{1}{2}},$$ where $\Theta_1',\Theta_2'$ are such that for some $\rho_1,\rho_2\in\G_\tau^{F,\Gamma}$, we have $\rho_1\cdot(B_1,\Theta_1)=(0,\Theta_1')$ as well as $\rho_2\cdot(B_2,\Theta_2)=(0,\Theta_2')$. To see that $d$ is well-defined on the image of $\Phi$, we need to show that $d$ does not depend on the choices of $\rho_1$ and $\rho_2$. Suppose we have another pair of elements $\rho_1',\rho_2'\in\G_\tau^{F,\Gamma}$ such that $\rho_1'\cdot(B_1,\Theta_1)=(0,\Theta_1'')$ as well as $\rho_2'\cdot(B_2,\Theta_2)=(0,\Theta_2'')$. Since both $\rho_1$ and $\rho_1'$ send $B_1$ to $0$, we must have that $\rho_1$ and $\rho_1'$ differ by a constant in $F$, say $f_1\rho_1=\rho_1'$. Similarly, we must have $\rho_2$ and $\rho_2'$ also differ by a constant in $F$, say $f_2\rho_2=\rho_2'$. Now, we have that $\Theta_1''=f_1\Theta_1'f_1^{-1}$ and $\Theta_2''=f_2\Theta_2'f_2^{-1}$. We compute the following: $$\inf_{f\in F}\int_{S^2/\Gamma}Re\langle f\Theta''_1f^{-1}-\Theta''_2,f\Theta''_1f^{-1}-\Theta''_2\rangle\omega_{vol})^{\frac{1}{2}}$$$$=\inf_{f\in F}\int_{S^2/\Gamma}Re\langle ff_1\Theta_1'f_1^{-1}f^{-1}-f_2\Theta_2'f_2^{-1},ff_1\Theta_1'f_1^{-1}f^{-1}-f_2\Theta_2'f_2^{-1}\rangle\omega_{vol})^{\frac{1}{2}}$$$$=\inf_{f\in F}\int_{S^2/\Gamma}Re\langle f_2^{-1}ff_1\Theta_1'f_1^{-1}f^{-1}f_2-\Theta_2',f_2^{-1}ff_1\Theta_1'f_1^{-1}f^{-1}f_2-\Theta_2'\rangle\omega_{vol})^{\frac{1}{2}}$$$$=\inf_{f\in F}\int_{S^2/\Gamma}Re\langle f\Theta'_1f^{-1}-\Theta'_2,f\Theta'_1f^{-1}-\Theta'_2\rangle\omega_{vol})^{\frac{1}{2}},$$ where the last equality holds as we are passing to the infimum. This shows that $d$ is well-defined. Hence, $\Phi$ is an isometry onto its image.
\end{proof}

\end{proof}

\subsection{Proof of Theorem \ref{main} Part II}

In this subsection, we prove the other direction of Theorem \ref{main}. That is, we show that the moduli space $\X_{\tilde{\zeta}}$ obtained by the gauge-theoretic construction is indeed equal to the 4-dimensional hyperk{\"a}hler ALE space $X_\zeta$ given by Kronheimer's construction in \cite{kronheimer}. To this end, we first prove the following lemma.

\begin{lem}
The complement of $X_\zeta$ contained in the gauge-theoretic moduli space $\X_{\tilde{\zeta}}$ is of higher codimension. 
\end{lem}

\begin{proof}
First, in the setup of \cite{kronheimer}, by a result of Kirwan \cite{kirwan} as cited also in \cite{kronheimer}, a stable orbit (closed and of maximal dimension) of $M$ under the action of $F^c$ contains a solution to the equation $\frac{i}{2}([\alpha,\alpha^*]+[\beta,\beta^*])=0$. Now, for any choice of $\zeta_1$, since $|\mu_1-\zeta_1|^2$ is proper on the $F^c$-orbit containing a solution to $\frac{i}{2}([\alpha,\alpha^*]+[\beta,\beta^*])=0$, and $F/T$ acts freely on a stable orbit, we have that the complex orbit also contains a solution to $\frac{i}{2}([\alpha,\alpha^*]+[\beta,\beta^*])=\zeta_1$. As the stable orbits are open and dense, the $F^c$-orbits not containing a solution to $\frac{i}{2}([\alpha,\alpha^*]+[\beta,\beta^*])=\zeta_1$ is of higher codimension.

On the other hand, a solution in $\X_{\tilde{\zeta}}$ that does not a priori come from a solution in $X_\zeta$ must have the form $(B,\Theta)$ with $B$ not $\G^{F,\Gamma}_\tau$-equivalent to $0$. Hence, it lies in a different connected component from the one containing the solutions coming from $X_\zeta$ and is contained in a non-stable orbit of $M$ when we identify the $F^c$-orbits of $M$ in \cite{kronheimer} with the $\G^{F,\Gamma}_{\tau,\C}$-orbits of $\mathcal{C}$ by Lemma \ref{psimap} and Remark \ref{psimaptau}. This tells us that the $\G^{F,\Gamma}_{\tau,\C}$-orbits that do not a priori contain a solution coming from Kronheimer's construction must be of higher codimension in the moduli space.

\end{proof}

\begin{proof}[Proof of main theorem Part II]

We want to argue that there are no additional solutions in the gauge-theoretic moduli space $\X_{\tilde{\zeta}}$ than the solutions coming from $X_\zeta$ in \cite{kronheimer}. We know if the gauge group acts freely at a solution, then it must come from Kronheimer's construction, by the previous lemma and dimension calculations. But by Lemma \ref{freeaction}, we know that the gauge group $\G_\tau^{F,\Gamma}$ acts freely on the space of solutions when $\zeta$ is not lying in $D_\xi$, which is precisely the assumption we have. Hence, all the solutions in $\X_{\tilde{\zeta}}$ must come from $X_\zeta$. Hence, they are equal, and $\Phi: X_\zeta\to \X_{\tilde{\zeta}}$ is an isometry. 
\end{proof}

We have concluded the proof of the main theorem, and we will end this section by providing the proof of Proposition \ref{moduli1}.

\begin{proof}[Proof of Proposition \ref{moduli1}]
This proof follows essentially the same arguments as those of the proof of Theorem \ref{main}. First, observe that \ref{holo} and \ref{eq1} reduce to $\frac{i}{2}([\alpha,\alpha^*]+[\beta,\beta^*])=\zeta_1$ when we set $B=0$. Hence, by Lemma \ref{psimap} and \ref{uniqueness1}, we again have that the space of solutions satisfying $\frac{i}{2}([\alpha,\alpha^*]+[\beta,\beta^*])=\zeta_1$ lies inside $\M(\Gamma, \tilde{\zeta}_1)$ as a subset. Since we assume that we are choosing $\tilde{\zeta}_1$ such that the action of the gauge group $\G^{F,\Gamma}$ on the space of solutions to  \ref{holo} and \ref{eq1} is free, we then know that $\M(\Gamma, \tilde{\zeta}_1)$ is smooth. Hence again, by Proposition \ref{dimcalc}, we know that there cannot be any additional solutions in $\M(\Gamma, \tilde{\zeta}_1)$, and we get the desired conclusion.
\end{proof}


\begin{thebibliography}{9}



\bibitem{atiyah-bott}
Atiyah, M. F., Bott, R., \emph{The Yang-Mills Equations over Riemann Surfaces}. Philosophical Transactions of the Royal Society of London. Series A, Mathematical and Physical Sciences, 308(1505), 523-615. (1983)


\bibitem{atiyah-hitchin}
Atiyah, M. F., Hitchin, N. J., \emph{The Geometry and Dynamics of Magnetic Monopoles}. Vol. 8. Princeton University Press. (2014)


\bibitem{boalch}
Boalch, P., \emph{Simply-Laced Isomonodromy Systems}. Publications math{\'e}matiques de l'IH{\'E}S 116: 1-68. (2012)


\bibitem{cherkis}
Cherkis, S. A., Kapustin, A.,  \emph{Singular Monopoles and Gravitational Instantons}. Communications in Mathematical Physics, 203, 713-728. (1999)


\bibitem{cherkis-k}
Cherkis, S. A., Kapustin, A., \emph{Singular Monopoles and Supersymmetric Gauge Theories in Three Dimensions}. Nuclear Physics B 525.1-2: 215-234. (1998)


\bibitem{cherkis-w}
Cherkis, S. A., Ward, R. S., \emph{Moduli of Monopole Walls and Amoebas}. Journal of High Energy Physics 2012.5: 1-37. (2012)


\bibitem{da-silva}
 Da Silva,  A. C.,  Da Salva, A. C., \emph{Lectures on Symplectic Geometry} (Vol. 2). Berlin: Springer.  (2008)
 
\bibitem{foscolo1}
Foscolo, L., \emph{Deformation Theory of Periodic Monopoles (with Singularities)}. Communications in Mathematical Physics, 341, 351-390.  (2016)
 
 \bibitem{foscolo2}
Foscolo,  L., \emph{ A Gluing Construction for Periodic Monopoles}. International Mathematics Research Notices, 2017(24), 7504-7550.  (2017)
 
 \bibitem{morgan}
  Friedman, R., Morgan, J. W., \emph{Gauge Theory and the Topology of Four-Manifolds} (Vol. 4). American Mathematical Society.  (1998)
 
 


\bibitem{griffiths-harris}
Griffiths, P., Harris, J., \emph{Principles of Algebraic Geometry}, Wiley Interscience, New York. (1994) 
https://doi.org/10.1002/9781118032527  


\bibitem{hitchin}
 Hitchin, N. J., \emph{The Self-Duality Equations on a Riemann Surface}, Proceedings of the London Mathematical Society, s3-55: 59-126.  (1987)

\bibitem{hitchin-supersymm}
 Hitchin, N. J., Karlhede, A.,  Lindstr{\"o}m, U.,  Ro\v{c}ek, M., \emph{Hyperk{\"a}hler Metrics and Supersymmetry}, Communications in Mathematical Physics, Comm. Math. Phys. 108(4), 535-589.  (1987)


\bibitem{kirwan}
 Kirwan, F. C., \emph{Cohomology of Quotients in Symplectic and Algebraic Geometry}. (MN-31), Volume 31, Princeton: Princeton University Press. (1985) https://doi.org/10.1515/9780691214566  

\bibitem{kobayashi2}
Kobayashi, S., \emph{Differential Geometry of Complex Vector Bundles} (Vol. 793). Princeton University Press.  (2014)


\bibitem{kobayashi1}
 Kobayashi, S.,  Nomizu, K., \emph{ Foundations of Differential Geometry} (Vol. 1, No. 2). New York, London.  (1963)



\bibitem{kronheimer}
 Kronheimer, P. B., \emph{The Construction of ALE Spaces as Hyperk{\"a}hler Quotients}, Journal of Differential Geometry, J. Differential Geom. 29(3), 665-683.   (1989)

\bibitem{kronheimer2}
Kronheimer, P. B., \emph{A Torelli-type Theorem for Gravitational Instantons}, Journal of Differential Geometry 29.3 (1989): 685-697.  (1989)

\bibitem{mckay}
 McKay, J., \emph{Graphs, Singularities, and Finite Groups}. Uspekhi Matematicheskikh Nauk, 38(3), 159-162.  (1983)


\bibitem{nahm}
Nahm, W., \emph{A Simple Formalism for the BPS Monopole}. Physics Letters B 90.4: 413-414. (1980)



\bibitem{sobolev-space}
 Roe, J., \emph{Elliptic Operators, Topology, and Asymptotic Methods}. CRC Press.  (1999)




\bibitem{salamon2}
 Salamon, D. A., \emph{Seiberg-Witten Invariants of Mapping Tori, Symplectic Fixed Points, and Lefschetz Numbers}, Turkish Journal of Mathematics: Vol. 23: No. 1, Article 7.  (1999) https://journals.tubitak.gov.tr/math/vol23/iss1/7






\end{thebibliography}
\end{document}